\documentclass[a4paper,reqno,11pt]{amsart}
\usepackage[utf8]{inputenc}
\usepackage{setspace}
\usepackage[left=2.5cm,right=2.5cm,top=3.5cm,bottom=2.9cm]{geometry}
\usepackage{amsthm,amssymb,amsmath,amsfonts,mathrsfs,amscd,yfonts, mathtools}
\usepackage{turnstile}
\usepackage[all]{xy}
\usepackage{latexsym}
\usepackage{longtable}
\usepackage[usenames,dvipsnames]{color}
\usepackage{xcolor}
\usepackage[textsize=tiny]{todonotes}
\usepackage{tikz-cd}
\usepackage{bbold}
\usepackage{mathdots}
\usepackage{comment}
\usepackage{bbm}
\usepackage{hyperref}
\hypersetup{
  colorlinks   = true, 
  urlcolor     = magenta, 
  linkcolor    = magenta, 
  citecolor   = cyan 
}

\theoremstyle{plain}
\newtheorem{theorem}{Theorem}[section]
\newtheorem{corollary}[theorem]{Corollary}
\newtheorem{lemma}[theorem]{Lemma}
\newtheorem{proposition}[theorem]{Proposition}
\newtheorem{propo}[theorem]{Proposition}

\theoremstyle{definition}
\newtheorem{definition}[theorem]{Definition}

\newtheorem{examplewr}[theorem]{Example}

\theoremstyle{remark}
\newtheorem{obswr}[theorem]{Observation}
\newtheorem{remarkwr}[theorem]{Remark}

\newenvironment{remark}{\begin{remarkwr}\begin{upshape}}{\end{upshape}\end{remarkwr}}

\DeclareMathOperator{\disc}{disc}
\DeclareMathOperator{\diag}{diag}

\DeclareMathOperator{\Gal}{Gal}

\DeclareMathOperator{\Hom}{Hom}
\DeclareMathOperator{\id}{id}
\DeclareMathOperator{\Ind}{Ind}
\DeclareMathOperator{\norm}{N}

\DeclareMathOperator{\rank}{rank}

\DeclareMathOperator{\Res}{Res}

\DeclareMathOperator{\Tr}{Tr}

\newcommand{\bb}{\mathbb}
\newcommand{\cl}{\mathcal}

\newcommand{\GL}{\mathrm{GL}}
\newcommand{\Mp}{{\mathrm{Mp}}}
\newcommand{\PGL}{{\mathrm{PGL}}}
\newcommand{\PU}{\mathrm{PU}}
\newcommand{\SL}{{\mathrm{SL}}}
\newcommand{\SO}{{\mathrm{SO}}}
\newcommand{\Sp}{{\mathrm{Sp}}}
\newcommand{\U}{\mathrm{U}}

\newcommand{\Norm}{\mathrm{N}}

\newcommand{\mat}[4]{\left(\begin{array}{cc}#1&#2\\#3&#4\end{array}\right)}

\title{Dihedral long root A-packets of $p$-adic $G_2$ via theta correspondence }
\author{Raúl Alonso, Qiao He, Mishty Ray, Martí Roset}

\address{R.A.: University of California, Santa Barbara, USA}
\email{raular@ucsb.edu}
\address{Q.H.: Columbia University, New York, USA}
\email{qh2275@columbia.edu}
\address{M.R.: University of Calgary, Calgary, Canada}
\email{mishty.ray@ucalgary.ca}
\address{M.R.: McGill University, Montreal, Canada}
\email{marti.rosetjulia@mail.mcgill.ca}


\begin{document}

\maketitle

\begin{abstract}
    We construct local Arthur packets associated with a dihedral long root $A$-parameter of a split reductive group of type $G_2$ over a nonarchimedian local field of characteristic zero. The construction relies on an exceptional correspondence for the pair $(\PU_3\rtimes \bb{Z}/2\bb{Z},G_2)$.
\end{abstract}
\tableofcontents{}

\section{Introduction}

Let $G$ be a connected reductive linear algebraic group over a number field $F$. In \cite{arthur1989unipotent, Arthur:unipotent-motivation}, J. Arthur has given a conjectural description of the constituents of the square integrable automorphic representations $\mathcal A_2(G)$ of $G$. The conjecture predicts that there is a decomposition 
\[
\mathcal A_2(G) = \bigoplus_\psi \mathcal A_{2, \psi}
\]
where each $\mathcal A_{2,\psi}$ is (to a first approximation) a near equivalence class of representations and the sum runs over equivalence classes of discrete A-parameters $\psi$. A-parameters are admissible maps
\[
\psi: L_F \times \SL_2(\bb{C}) \to \prescript{L}{}{G},
\]
where $L_F$ denotes the conjectural Langlands group of $F$ and $\prescript{L}{}{G}=\hat{G}(\bb{C})\rtimes W_F$ denotes the L-group of $G$. Here $\hat{G}(\bb{C})$ is the complex dual group of $G$ and $W_F$ is the Weil group of $F$. We say that $\psi$ is discrete if the component group $S_\psi$, defined as the centralizer of the image of $\psi$ in $\hat{G}(\bb{C})$ modulo the center of $\hat{G}(\bb{C})$, is finite. Fix such an A-parameter $\psi$. Arthur's conjecture gives a more precise description of the constituents of $\mathcal A_{2, \psi}$. It first describes the local components of the representations appearing in $\mathcal A_{2,\psi}$, via the so-called \emph{local A-packets}, and then determines which combinations of such local representations appear globally, using the \emph{global A-packets} and the \emph{multiplicity formula}.

Let $v$ be a place of $F$, let $F_v$ be the completion of $F$ at $v$ and denote by $L_{F_v}$ the group $W_{F_v} \times \mathrm{SU}_2(\bb{C})$, where $W_{F_v}$ is the Weil group of $F_v$. We can pre-compose $\psi$ with a fixed embedding $L_{F_v} \xhookrightarrow{} L_F$ to obtain the local A-parameter $\psi_v: L_{F_v} \times \SL_2(\bb{C}) \to \prescript{L}{}{G}$. Define the local component group of $\psi_v$ as
\[
S_{\psi_v} = \pi_0\left(Z_{\hat{G}(\bb{C})}(\mathrm{Im}(\psi_{v})) Z(\hat{G}(\bb{C})) / Z(\hat{G}(\bb{C}))  \right),
\]
where $Z_{\hat{G}(\bb{C})}(\mathrm{Im}(\psi_{v}))$ denotes the centralizer in $\hat{G}(\bb{C})\subset \prescript{L}{}{G}$ of the image of $\psi_v$, $Z(\hat{G}(\bb{C}))$ denotes the center of $\hat{G}(\bb{C})$ and $\pi_0$ denotes the group of connected components. Arthur predicted that to each irreducible representation $\eta_v$ of $S_{\psi_v}$, we can attach a unitarizable finite length (possibly reducible, possibly zero) representation $\pi_{\eta_v}$ of $G(F_v)$. The collection
\[
A_{\psi_v} = \left\{\pi_{\eta_v} \mid \eta_v \in \mathrm{Irr}(S_{\psi_v})  \right\}
\]
is the local A-packet associated to $\psi_v$. There are several requirements on the representations in $A_{\psi_v}$. One of them is that if we let $1_v$ denote the trivial representation of $S_{\psi_v}$, then for all but finitely many $v$, $\pi_{1_v}$ is the unramified representation with Satake parameter 
\[
s_{\psi_v} = \psi_v\left(\Phi_v \times \mat{q_v^{-1/2}}{0}{0}{q_v^{1/2}} \right).
\]
Here $\Phi_v$ denotes a geometric Frobenius element at $v$ and $q_v$ is the size of the residue field at $v$. Given the local A-packets, we define the global A-packet associated to $\psi$ as 
\[
A_\psi = \left\{ \pi = \otimes_v' \pi_{\eta_v} \mid  \pi_{\eta_v} \in A_{\psi_v} \text{ for all } v \text{ and } \eta_v = 1_v \text{ for all but finitely many } v \right\}.
\]
Note that this is a set of nearly equivalent representations of $G(\bb{A}_F)$, which  is indexed by irreducible representations of $S_{\psi, \bb{A}_F} = \prod_v S_{\psi_v}$. For a given $\eta = \otimes_v \eta_v$, where $\eta_v \in \mathrm{Irr}(S_{\psi_v})$ and $\eta_v = 1_v$ for all but finitely many $v$, set $\pi_\eta = \otimes_v' \pi_{\eta_v}$. 
Arthur constructed a quadratic character $\epsilon_\psi$ of $S_\psi$ and used it to determine the multiplicity of each representation $\pi_\eta \in A_\psi$ appearing in $\mathcal A_{2, \psi}$. This yields the multiplicity formula 
\[
\mathcal A_{2, \psi} = \bigoplus_\eta m_\eta \pi_\eta, \ \text{ where } \ m_\eta = \frac{1}{\# S_\psi} \left( \sum_{s \in S_\psi} \epsilon_\psi(s) \eta(s)  \right).
\]
It is worth mentioning that the conjecture also predicts when all the representations in $\mathcal A_{2,\psi}$ are tempered, which should occur when the A-parameter $\psi$ is tempered (if $\psi$ restricted to $\SL_2(\bb{C})$ is trivial). On the other hand, nontempered A-parameters lead to local A-packets that can contain both tempered and nontempered representations. An important feature about nontempered A-parameters is that they usually factor through subgroups $^L\!H \subset \prescript{L}{}{G}$. It is then expected that we can construct the representations in the (local and global) A-packets associated to $\psi$ from representations of $H$. 

From now on suppose that $G$ is a split exceptional group of type $G_2$ over $F$. Note that $\prescript{L}{}{G}$ can be replaced by $G(\bb{C})$ and that there are $4$ different conjugacy classes of morphisms $\SL_2(\bb{C}) \to G(\bb{C})$ corresponding to the $4$ nontrivial unipotent conjugacy classes in $G(\bb{C})$. They give rise to $4$ families of nontempered A-parameters $\psi$ for $G_2$: if $\psi_{\vert \SL_2(\bb{C})}$ corresponds to the regular orbit, if $\psi_{\vert \SL_2(\bb{C})}$ corresponds to the subregular orbit, if $\psi_{\vert \SL_2(\bb{C})}$ gives the short root $\SL_2$ in $G_2$ and if $\psi_{\vert \SL_2(\bb{C})}$ gives the long root $\SL_2$ in $G_2$. For the first three mentioned families of nontempered A-parameters, Arthur's conjecture has been verified; see \cite{GGJ} and \cite{GG}. This work is part of a larger project initiated at the 2022 Arizona Winter School which aims to verify Arthur's conjecture for the so-called \emph{dihedral long root A-parameters}, a type of nontempered A-parameters for $G_2$ which belong to the fourth mentioned family. In particular, in this paper we construct the local nonarchimedean A-packets associated to dihedral long root A-parameters. We proceed to briefly introduce dihedral long root A-parameters and then we summarize how the corresponding local A-packets are constructed. For that we use a theta lift from $\PU_3 \rtimes \bb{Z}/2\bb{Z}$ to $G_2$ arising from the exceptional theta correspondence for $\left( \PU_3 \rtimes \bb{Z}/2\bb{Z} \right) \times G_2$ studied in \cite{endoscopiclifting}, \cite{GS21} and especially in \cite{BS}.

Let $K$ be a quadratic extension of $F$ and $c$ be the nontrivial element in the Galois group $\Gal(K/F)$. Denote by $\chi$ a character of $\bb{A}_K^\times/K^\times$ which is conjugate symplectic, \textit{i.e.}, $\chi_{\vert \bb{A}_F^\times} = \omega_{K/F}$, where $\omega_{K/F}$ is the quadratic character attached to the extension $K/F$. Moreover, suppose that $\chi^c \neq \chi$. Note that we can regard $\chi$ as a character of $W_K$. Let $\tau = \otimes_v' \tau_v$ be the representation of $\GL_2(\bb{A}_F)$ obtained from $\chi$ by automorphic induction, namely $\tau$ is the automorphic representation with $L$-parameter
\[
\rho_\tau = \Ind_{W_K}^{W_F} \chi.
\]
Note that the central character of $\tau$ is $\omega_{K/F}\cdot\chi_{\vert \bb{A}_F^\times} = 1$. Therefore, we regard $\rho_\tau$ as a representation of $\PGL_2$. Denote by $\SL_{2,l}$ (resp. $\SL_{2,s}$) the long (resp. short) root $\SL_2$ inside $G_2$. We define the long root A-parameter of $G$ associated with $\tau$ as
\[
\psi_{\tau, l}: L_F \times \SL_2(\bb{C}) \twoheadrightarrow{} W_F \times \SL_2(\bb{C}) \xrightarrow{\rho_\tau \times \id} \SL_{2,s}(\bb{C}) \times_{\mu_2} \SL_{2,l}(\bb{C}) \subset G(\bb{C}).
\]
There is a subgroup $\SL_{3,l}(\bb{C})\subset G(\bb{C})$ corresponding to the root system of type $A_2$ formed by the six long roots of $G$. Its normalizer inside $G(\bb{C})$ is isomorphic to $\SL_{3,l}(\bb{C}) \rtimes \bb{Z}/2\bb{Z}$. This is at the same time isomorphic to the L-group of $\PU_3$, the projective unitary group in three variables associated to the extension $K/F$. 
\[N_{G(\bb{C})}(\SL_{3,l}(\bb{C})) \simeq \SL_{3,l}(\bb{C}) \rtimes \bb{Z}/2\bb{Z} = ^L\!\PU_3.\]
In Section~\ref{sec: A-packets} we verify that $\psi_{\tau, l}$ can be conjugated to take values in that normalizer. We can therefore define the restriction 
\[
\psi_{\chi}: L_F \times \SL_2(\bb{C}) \to  ^L\!\PU_3.
\]
As discussed above, the fact that the A-parameter $\psi_{\tau, l}$ factors through an A-parameter of $\PU_3$ suggests that we can obtain the local A-packets for $\psi_{\tau, l}$ from the local A-packets for $\psi_{\chi}$. We explain this phenomenon when $v$ is a nonarchimedean place of $F$ which is nonsplit in $K$, as this is the most interesting case.

Fix a place of $K$ above $v$, denote by $K_v$ the completion of $K$ at that place and by $\chi_v$ the corresponding local component of $\chi$. The local component group of $\psi_{\chi,v}$ has two elements. Therefore, the local A-packet of $\PU_3$ associated to $\psi_{\chi_v}$ has the form 
\[
A_{\psi_{\chi,v}} = \left\{ \sigma_v^+, \sigma_v^- \right\},
\]
where $\sigma_v^+$ (resp. $\sigma_v^-$) corresponds to the trivial (resp. nontrivial) representation of the local component group. These representations can be obtained as theta lifts using the classical local theta correspondence between the unitary groups $\U_1$ and $\U_3$, as we explain in Section~\ref{sec: A-packets}. In particular, $\sigma_v^+$ is nontempered and $\sigma_v^-$ is supercuspidal. 

On the other hand, the local component group of $\psi_{\tau, l, v}$ has either two elements if $\chi_v^2 \neq 1$ or one element otherwise. Therefore, we expect that the local A-packet associated to $\psi_{\tau, l}$ has the form 
\[
A_{\psi_{\tau, l, v}} = \begin{cases}
\{\pi_v^+, \pi_v^-\}  \text{ if } \chi_v^2 \neq 1,\\
\{\pi_v^+\} \text{ if } \chi_v^2 = 1.
\end{cases}
\]
Here, $\pi_v^+$ corresponds to the trivial representation of $S_{\psi_{\tau, l , v}}$ and, when the local component group has two elements, $\pi_v^-$ is the representation corresponding to the nontrivial representation of $S_{\psi_{\tau, l , v}}$. Denote by $\Theta_{\PU_3}$ the exceptional big theta lift from $\PU_3$ to $G$ considered in \cite{BS} and by $\theta_{\PU_3}$ the corresponding small theta lift. Let $Q_1$ be the non-Heisenberg parabolic subgroup of $G(F_v)$, which has a Levi subgroup isomorphic to $\GL_2(F_v)$, and denote by $i_{Q_1}^{G}$ the normalized parabolic induction from $Q_1$ to $G(F_v)$. The main result of this paper is the following. 

\begin{theorem}
    Let $v$ be a non-archimedean place of $F$ which is nonsplit in $K$. Let $A_{\psi_{\chi,v}} = \{ \sigma_v^+, \sigma_v^-\}$ be the local A-packet associated to the A-parameter $\psi_{\chi}$ of $\PU_3$, where $\sigma_v^+$ is nontempered and $\sigma_v^-$ is supercuspidal. Define $\pi_v^{\pm} = \theta_{\PU_3}(\sigma^{\pm})$. Then:
	\begin{enumerate}
		\item The representation $\pi_v^+$ is the unique nonzero irreducible quotient of $i_{Q_1}^G(\lvert \det\rvert_{F_v}^{1/2}\tau_v)$, where $\lvert \ \rvert_{F_v}$ denotes the normalized absolute value on $F_v$.
		\item If $\chi_v^2 \neq 1$, the representation $\pi_v^-$ is nonzero, irreducible and tempered. 
		\item If $\chi_v^2 = 1$, the representation $\pi_v^-$ is zero. 
	\end{enumerate}
\end{theorem}

Note that as a consequence of the main theorem, we obtain a natural construction of the elements of the local A-packet $A_{\psi_{\tau, l ,v}}$ as the nonzero lifts of the elements of $A_{\psi_\chi,v}$ via $\theta_{\PU_3}$. Using a similar strategy, we construct the local A-packet $A_{\psi_{\tau, l, v}}$ when $v$ is a non-archimedean place of $F$ that is split in $K$ (see Theorem \ref{split theta lift to g2}).  

The first point of the main theorem is a consequence of the work of \cite{BS} on lifts of nontempered representations. To prove the second point we compute Fourier--Jacobi periods of the minimal representation used to define $\theta_{\PU_3}$ to obtain the following non-vanishing criterion: if the contragradient of a local representation of $\PU_3$, when restricted to a suitable two-variable unitary subgroup, has a quotient with trivial central character, then this representation has nonzero theta lift to $G$. We then verify that $\pi_v^-$ satisfies this condition if $\chi_v^2 \neq 1$ using a see-saw argument. In fact, this same reasoning shows that $\pi_v^+$ is nonzero in all cases, giving an alternative proof of the nonvanishing part in the first point of the theorem. The key ingredient to prove the third point of the theorem is that the twisted coinvariant spaces for $\pi_v^-$ corresponding to generic characters of the unipotent of the Heisenberg group of $G(F_v)$ vanish. This is proven using that for every such character we have:  
\begin{itemize}
    \item The explicit description of $\pi_v^+$ allows to verify that the twisted coinvariant space for $\pi_v^+$ is $1$-dimensional.
    \item The twisted coinvariant space for $\pi_v^+ \oplus \pi_v^-$ can be related to a sum of toric periods for $\sigma_v^+ \oplus \sigma_v^-$. Since the representations $\sigma_v^{\pm}$ are theta lifts of characters in $\U_1$, the non-vanishing of these periods can be expressed in terms of local epsilon factors, as it is done in \cite{BFGYYZ}. Moreover, in [\emph{loc.\,cit.}], it is proven that exactly one of these toric periods contributes to the sum with precisely dimension 1.     
\end{itemize}


We expect that a similar construction should yield the local A-packets in the archimedean case, but the theta correspondence in the archimedean case has not been fully analysed yet. We refer the reader to \cite{BHLS}, where a definition of the local A-packets in this setting, which can be related to theta lifts in some cases, is proposed.  

The $p$-adic construction presented in this article is used in [\emph{loc.\,cit.}] to construct global dihedral long root A-packets and, under the hypothesis that $L(\chi, 1/2) \neq 0$ and certain conditions regarding the archimedean construction, verify the multiplicity formula.

We conclude the introduction with notes about the organisation of this paper. In Section \ref{section: Theta correspondence for unitary groups}, we establish notation for unitary groups and recall essential results about the classical local theta correspondence between $\U_1$ and $\U_3$. In Section \ref{section: Theta correspondence PU3 times G2}, we introduce the exceptional group of type $G_2$, the local exceptional theta correspondence between this group and $\PU_3\rtimes\bb{Z}/2\bb{Z}$, and review results about this correspondence that we will need later. In Section \ref{sec: A-packets}, we describe the construction of the local A-packets associated to $\psi_{\tau,l}$ for all finite primes. In Section \ref{section: nonvanishing of theta lifts}, we prove nonvanishing of the representation $\pi_v^-$ in the case $\chi_v^2 \neq 1$. Finally, in Section \ref{section: vanishing of theta lifts}, we prove the vanishing of $\pi_v^-$ in the case $\chi_v^2 = 1$.

\subsection{Acknowledgements} This project was proposed by Wee Teck Gan as part of the Arizona Winter School 2022, and we would like, first and foremost, to thank him for introducing us to this subject and for all his continued support during the development of the project. Warm thanks are also due to Petar Baki\'c for his extremely valuable help. We would also like to thank Hung Chiang and Yu-Sheng Lee for their participation at the outset of this project, as well as the rest of our fellow participants in the Arizona Winter School project. We also want to thank Eric Chen, Sam Mundy and Marco Sangiovanni Vicentelli for several helpful conversations related to this work. Last but not least, we would like to thank the organizers of the Arizona Winter School 2022 for providing us with the opportunity to engage in this project and to work for a week in a wonderful and stimulating environment. 

Qiao He is partially supported by a graduate school grant of UW-Madison. Mishty Ray is supported by the Graduate Assistantship (Teaching) fund and the Eric Milner scholarship at the University of Calgary. Mart\'i Roset received the support of a fellowship from la Caixa Foundation (ID 100010434). The fellowship code is LCF/BQ/EU21/11890132.

\section{Theta correspondence for unitary groups}\label{section: Theta correspondence for unitary groups}
In this section, we work out the local theta correspondence for the dual pair $(\U_1,\U_3)$. We start by recalling some general structure theory for unitary groups in Section~\ref{groupsandreps}, followed by a description of the theta correspondence for general unitary groups in Section~\ref{subsection: theta lifting for general unitary groups}. We then apply this theory to the dual pair $(\U_1,\U_3)$ in Section~\ref{subsection: Theta correspondence for U1 times U3} for nonsplit places, and $(\GL_1,\GL_3)$ in Section~\ref{subsection: Theta correspondence for GL1 times GL3} for split places.

\subsection{Definitions}\label{groupsandreps}

Let $F$ be a nonarchimedian local field of characteristic zero and let $K$ be a quadratic field extension with $\Gal(K/F)=\langle c \rangle$. Let $\omega_{K/F}$ be the nontrivial quadratic character of $F^{\times}/N_{K/F}(K^{\times})$. A finite-dimensional Hermitian (resp. skew-Hermitian) space over $K$ is a finite dimensional vector space over $K$ equipped with a nondegenerate sesquilinear form $\langle \, , \, \rangle$ satisfying $\langle v, w\rangle^c=\langle w,v\rangle$ (resp. $\langle v, w\rangle^c=-\langle w,v\rangle$).
We adopt the convention that sesquilinear forms are linear on the first variable and conjugate-linear on the second variable.

For each positive integer $n$, there are two isomorphism classes of Hermitian spaces of dimension $n$ over $K$. Given a Hermitian space of dimension $n$ over $K$, its isomorphism class is determined by an invariant known as the discriminant, which we now define.

Let $V$ be a Hermitian space of dimension $n$ over $K$. Let $\{v_1,v_2,\ldots,v_n\}$ be a $K$-basis of $V$ and let $\Phi=(\langle v_i,v_j \rangle)$ be the matrix of inner products of the basis elements. Then, the discriminant of $V$, which we denote by $\disc(V)$, is defined by
\[
\disc(V)=(-1)^{n(n-1)/2}\det(\Phi) \in F^{\times}/N_{K/F}(K^{\times}).
\]
The sign character that classifies $V$ is given by 
\begin{equation}\label{epsilon v}
    \epsilon(V)=\omega_{K/F}(\disc(V))\in\lbrace\pm 1\rbrace.
\end{equation}

If $W$ is a skew-Hermitian space of dimension $m$, we can define the discriminant using the same procedure, but now we get
\[
\disc(W)\in \delta^{m}F^{\times}/N_{K/F}(K^{\times}),
\]
where $\delta$ denotes any trace-zero element in $K^\times$. We can again attach a sign to $W$ by setting
\begin{equation}\label{epsilon w}
    \epsilon(W)=\omega_{K/F}(\delta^{-m}\disc(W))\in\lbrace\pm 1\rbrace,
\end{equation}
but note that, if $m$ is odd, this definition depends on our choice of $\delta$.

For an $n$-dimensional Hermitian or skew-Hermitian space $V$, we denote by $\underline{\U}(V)$ the corresponding unitary group. This is an algebraic group over $F$. Choose a $K$-basis of $V$ and let $\Phi$ be the matrix of the (skew-)Hermitian form on $V$ with respect to this basis. Then, for any $F$-algebra $R$, the $R$-points of $\underline{\U}(V)$ can be described as
\[
\underline{\U}(V)(R)=\lbrace g\in \GL_n(K\otimes_F R) \;\colon\; g\Phi g^\dagger=\Phi\rbrace.
\]
We are primarily interested in the group of $F$-points of $\underline{\U}(V)$, which we denote by $\U(V)$.

Observe that a Hermitian form becomes skew-Hermitian after multiplication by a trace-zero element in $K^\times$ and vice versa, without changing the associated unitary groups. Therefore, from now on we focus on Hermitian spaces.

Let $V$ be an $n$-dimensional Hermitian space with Hermitian form $\langle\, , \,\rangle$. Let $a\in F^\times$ and let $V^a$ denote the Hermitian space with the same underlying space $V$ equipped with the Hermitian form $a\langle\, , \,\rangle$. From the above description of the associated unitary groups, it is clear that $\underline{\U}(V^a)= \underline{\U}(V)$.


If $n=2m$ for a positive integer $m$, we have that $V\simeq V^a$ from the definition of the discriminant, and it can be proved that non-isomorphic Hermitian spaces yield non-isomorphic unitary groups. When $m=1$ and $\epsilon(V)=1$, we can choose a basis $\{e_1,e_2\}$ so that $V=Ke_1 \oplus Ke_2$ and 
    \[\langle e_1,e_1 \rangle = \langle e_2,e_2 \rangle =0 \text{ and } \langle e_1,e_2 \rangle =1.\]
We call this $2$-dimensional space the hyperbolic plane and denote it by $\mathbb{H}$. More generally, when $\epsilon(V)=1$, then $V \simeq \mathbb{H}^m$. We say that such a $V$ is split and the 
corresponding unitary group is quasi-split.  In this case, we can choose a basis for $V$ for which the Hermitian form is given by
\[
\Phi=
\begin{pmatrix}
0 & 0 & \cdots & 0 & 1 \\
0 & 0 & \cdots & 1 & 0 \\
\vdots & \vdots & \iddots & \vdots & \vdots \\
0 & 1 & \cdots & 0 & 0 \\
1 & 0 & \cdots & 0 & 0
\end{pmatrix}.
\]
Once we fix such a basis, the subgroup of upper-triangular matrices in $\underline{\U}(V)$ defines a Borel subgroup, which we sometimes refer to as the standard Borel subgroup. Let $B'$ denote the $F$-points of this Borel subgroup and let $T' \subset B'$ denote the $F$-points of the maximal torus consisting of diagonal matrices. Then $T'$ consists of the elements of the form
\[
t(a_1, \ldots,a_m)=\diag\left(a_m,...,a_{1},(a_1^c)^{-1},...,(a_m^c)^{-1}\right), \quad\text{with } a_1,\ldots,a_m\in K^\times,
\]
so we get an identification $T'\simeq (K^\times)^m$. If $\epsilon(V)=-1$, then $V$ is isomorphic to the orthogonal direct sum of $m-1$ hyperbolic planes and an anisotropic two-dimensional Hermitian space. In this case, $V$ is nonsplit and the corresponding unitary group is not quasi-split.

If $n=2m+1$, then for an element $a\in F^\times$ which is not a norm from $K^\times$, the Hermitian spaces $V$ and $V^a$ are not isomorphic. It follows that there is only one isomorphism class of unitary groups in $n$-variables, which are always quasi-split. We denote any element in this isomorphism class by $\underline{\U}_n$, and its $F$-points by $\U_n$. Then $V$ is isomorphic to the orthogonal direct sum of $m$ hyperbolic planes and a line. We can choose a basis of $V$ for which the corresponding Hermitian form is given by
\[
\Phi=
\begin{pmatrix}
0 & 0 & \cdots & 0 & b \\
0 & 0 & \cdots & b & 0 \\
\vdots & \vdots & \iddots & \vdots & \vdots \\
0 & b & \cdots & 0 & 0 \\
b & 0 & \cdots & 0 & 0
\end{pmatrix},
\]
with the class of $b$ in $F^\times / \norm_{K/F}(K^\times)$ determined by the discriminant of $V$. Once we fix such a basis, the subgroup of upper-triangular matrices in $\underline{\U}(V)$ defines a Borel subgroup, which we sometimes refer to as the standard Borel subgroup. Let $B'$ denote the $F$-points of this Borel subgroup and let $T' \subset B$ denote the $F$-points of the maximal torus consisting of diagonal matrices. Then $T'$ consists of the elements of the form
\[
t(a_0, a_1, \ldots,a_m)=\diag\left(a_m,...,a_{1},a_0,(a_1^c)^{-1},...,(a_m^c)^{-1}\right), \quad\text{with } a_0\in K^1,\, a_1,\ldots,a_m\in K^\times,
\]
so we get an identification $T'\simeq (K^\times)^m\times K^1$.

\subsection{Theta correspondence for unitary groups}\label{subsection: theta lifting for general unitary groups}

Let $F$ be a nonarchimedian local field of characteristic zero and let $K/F$ be a quadratic field extension. Let $V$ be a Hermitian space over $K$ of dimension $n$ and let $W$ be a skew-Hermitian space over $K$ of dimension $m$. We can regard $V\otimes_{K}W$ as a vector space over $F$ equipped with the symplectic form
\[
\frac{1}{2}\mathrm{Tr}_{K/F}(\langle \ , \ \rangle_V\otimes_K\langle \ , \ \rangle_W).
\]
Let $\Sp(V \otimes_K W)$ be the symplectic group associated with this symplectic space. Then we have a natural map
\[
\iota : \U(V)\times \U(W)\longrightarrow \mathrm{Sp}(V\otimes_{K}W),
\]
and $\U(V)$ and $\U(W)$ form a reductive dual pair inside $\Sp(V \otimes_K W)$. The aim of this section is to describe the theta correspondence for this pair.

Fix a nontrivial additive character $\psi : F \rightarrow \bb{C}^\times$. Let $\Mp(V\otimes_{K}W)$ be the metaplectic group associated with the symplectic space $V\otimes_{K}W$, which for us will be an $S^1$-cover of the symplectic group $\Sp(V\otimes_{K} W)$. Let $\omega_{\psi}$ denote the Weil representation of $\Mp(V\otimes_{K}W)$ corresponding to the character $\psi$.

After fixing two characters $\chi_{V}$, $\chi_{W}$ of $K^\times$ such that
\[
{\chi_{V}}_{\vert F^\times} = \omega_{K/F}^{n} \text{ and } {\chi_{W}}_{\vert F^\times} = \omega_{K/F}^{m},
\] 
the work of Kudla \cite{K} provides a morphism 
\[
\tilde{\iota}_{\chi_{V}, \chi_{W}, \psi}: \U(V) \times \U(W) \longrightarrow \mathrm{Mp}(V\otimes_{K} W)	
\]
lifting the natural map $\iota: \U(V) \times \U(W) \to \mathrm{Sp}(V \otimes_{K} W)$. Hence, we can consider the representation of $\U(V) \times \U(W)$ obtained as the pullback of $\omega_{\psi}$ by $\tilde{\iota}_{\chi_V, \chi_W, \psi}$. 
We will denote this representation by $\Omega_{V, W, \chi_V, \chi_W, \psi}$, or simply by $\Omega$. We use this  representation to describe the theta correspondence between $\U(W)$ and $\U(V)$.
\begin{definition}
	Let $\pi$ be an irreducible smooth representation of $\U(W)$. The maximal $\pi$-isotypic quotient of $\Omega$ is 
	\[
	\Omega / \bigcap_{f \in \Hom_{\U(W)}(\Omega, \pi)} \ker(f).
	\]
	This is a representation of $\U(V) \times \U(W)$ and can be written as $\Theta(\pi) \boxtimes \pi$, where $\Theta(\pi)$ is a smooth representation of $\U(V)$. We denote by $\Theta_{V,W,\chi_V, \chi_W,\psi}(\pi)$, or simply by $\Theta(\pi)$, the representation of $\U(V)$ obtained from $\pi$ following this procedure. It is called the \textit{big theta lift} of $\pi$.
\end{definition}

The following theorem was a conjecture of Howe \cite{Howe}. For odd residue characteristic, a proof was given by Waldspurger \cite{Waldspurger}. The assumption on the residue characteristic was removed by Gan--Takeda \cite{GT}.  
\begin{theorem}[Howe duality theorem]
	Let $\pi,\pi' \in \mathrm{Irr}(\U(W))$.  
	\begin{enumerate}
		\item If $\Theta(\pi)$ is nonzero, it has a unique irreducible quotient. We denote it by $\theta_{V, W, \chi_V, \chi_W, \psi}(\pi)$, or simply by $\theta(\pi)$. It is called the small theta lift of $\pi$.
		\item If $\theta(\pi) \simeq \theta(\pi') \neq 0$, then $\pi \simeq \pi'$.
	\end{enumerate}
\end{theorem}

It is useful to study theta lifts by considering certain families of Hermitian spaces called Witt towers. Recall that $\bb{H}$ denotes the hyperbolic plane. We say that two Hermitian spaces $V$ and $V'$ belong to the same Witt tower if there exist integers $k,l\geq 0$ such that
\[
V\oplus \bb{H}^k\simeq V'\oplus \bb{H}^l.
\]
Let $V_1$ and $V_1'$ denote two non-isomorphic 1-dimensional Hermitian spaces and let $V_2$ denote a 2-dimensional anisotropic Hermitian space. There are two Witt towers of even-dimensional Hermitian spaces
\[
\cl{W}'_0=\lbrace \bb{H}^k\;\colon\; k\geq 0\rbrace,\quad \cl{W}_0=\lbrace V_2\oplus \bb{H}^k\;\colon\; k\geq 0\rbrace,
\]
and two Witt towers of odd-dimensional Hermitian spaces
\[
\cl{W}'_1=\lbrace V_1'\oplus \bb{H}^k\;\colon\; k\geq 0\rbrace,\quad \cl{W}_1=\lbrace V_1\oplus \bb{H}^k\;\colon\; k\geq 0\rbrace.
\]

We continue to denote by $W$ a fixed $m$-dimensional skew-Hermitian space. Fix a character $\chi_W$ satisfying the condition stated above, i.e., such that ${\chi_{W}}_{\vert F^\times} = \omega_{K/F}^{m}$. Since the parity of the dimension is the same for elements in a fixed Witt tower, we can choose the same splitting character $\chi_V$ to define the corresponding theta lifts. Thus, we fix characters $\chi_{\text{odd}}$ and $\chi_{\text{even}}$ such that ${\chi_{\text{odd}}}_{\vert F^\times} = \omega_{K/F}$ and ${\chi_{\text{even}}}_{\vert F^\times} = 1$.

For an irreducible smooth admissible representation $\pi$ of $\U(W)$, we make the following definitions:
\begin{align*}
    n_{\cl{W}_0'}(\pi)&=\min \lbrace\dim V \;\colon\; V\in\cl{W}_0' \text{ and } \Theta_{V,W,\chi_{\text{even}},\chi_W,\psi}(\pi)\neq 0 \rbrace; \\
    n_{\cl{W}_0}(\pi)&=\min \lbrace\dim V \;\colon\; V\in\cl{W}_0 \text{ and } \Theta_{V,W,\chi_{\text{even}},\chi_W,\psi}(\pi)\neq 0 \rbrace; \\
    n_{\cl{W}_1'}(\pi)&=\min \lbrace\dim V \;\colon\; V\in\cl{W}_1' \text{ and } \Theta_{V,W,\chi_{\text{odd}},\chi_W,\psi}(\pi)\neq 0 \rbrace; \\
    n_{\cl{W}_1}(\pi)&=\min \lbrace\dim V \;\colon\; V\in\cl{W}_1 \text{ and } \Theta_{V,W,\chi_{\text{odd}},\chi_W,\psi}(\pi)\neq 0 \rbrace.
\end{align*}
For the definition of $n_{\cl{W}_0'}(\pi)$, we consider that, in the case $V=0$, the theta lift of $\pi$ is nonzero if and only if $\pi$ is the 1-dimensional representation defined by the character $\chi_{\text{even}}\circ i\circ\det_W $, where $i$ denotes the inverse of the isomorphism $K^\times/F^\times\xrightarrow{\simeq} K^1$ defined by $x\mapsto x/x^c$ and $\det_W$ denotes the natural determinant map on $\U(W)$.

The following result is a special case of \cite[Thm.~1.10]{SZ} (see also the discussion preceding the statement of the theorem).
\begin{theorem}\label{thm: conservation law}
Let $\pi$ be an irreducible smooth representation of $\U(W)$. Then:
\begin{align*}
    n_{\cl{W}_0'}(\pi)+n_{\cl{W}_0}(\pi)=2m+2; \\
    n_{\cl{W}_1'}(\pi)+n_{\cl{W}_1}(\pi)=2m+2.
\end{align*}
Moreover, for any Witt tower $\cl{W}$ and for any $V\in\cl{W}$ with $\dim V\geq n_{\cl{W}}(\pi)$, the corresponding theta lift $\Theta(\pi)$ is nonzero. The same results hold if we interchange the role of Hermitian and skew-Hermitian spaces.
\end{theorem}

\subsection{Theta correspondence for $\U_1\times \U_3$}\label{subsection: Theta correspondence for U1 times U3}

We keep the definitions and the notation from the previous subsection. Let $\gamma$ be a conjugate-symplectic character of $K^\times$, i.e., such that $\gamma_{\vert F^\times} = \omega_{K/F}$. Then, we can make the following choice of splitting characters:
\[
\chi_V = \gamma^{n}, \ \chi_W = \gamma^{m}.
\]
We denote by $\Omega_{V,W,\gamma,\psi}$ the pullback to $\U(V)\times \U(W)$ of the Weil representation $\omega_\psi$ obtained from this choice of splitting characters. For an irreducible smooth representation $\pi$ of $\U(W)$, we denote by $\Theta_{V,W,\gamma,\psi}(\pi)$ the corresponding theta lift. Due to the choice of splitting characters, the theta correspondence preserves central characters.

Assume now that $W$ has dimension $m=1$ and $V$ has dimension $n=3$. Let $V_1$ denote the 1-dimensional Hermitian space in the Witt tower of $V$.

\begin{proposition}\label{thetaliftinWitttower}
	Let $\mu$ be a character of $\U(W)$. Then,
	\begin{enumerate}
		\item if $\Theta_{V_1, W, \gamma, \psi}(\mu \gamma_{\vert K^1}^{-1}) = 0$, then $\Theta_{V, W,\gamma,\psi}(\mu)$ is a nonzero irreducible supercuspidal representation of $\U(V)$;
		\item if $\Theta_{V_1, W, \gamma, \psi}(\mu \gamma_{\vert K^1}^{-1}) \neq 0$, then $\Theta_{V, W, \gamma, \psi}(\mu)$ is a nonzero, irreducible but not supercuspidal representation of $\U(V)$. Moreover, $\Theta_{V, W,\gamma,\psi}(\mu)$ is a quotient of 
		\[
		i_{B'}^{\U(V)}\left(\gamma \lvert \  \rvert_{K}^{1/2} \otimes \mu \gamma_{\vert K^1}^{-1}\right).
		\]
	\end{enumerate}
\end{proposition}
\begin{proof}
This follows from \cite[Théorème principal, p.~69]{MVW}. Indeed, any character of $\U(W)$ is a supercuspidal representation. The appearance of the character $\gamma_{\vert K^1}^{-1}$ in our statements follows from our different choice of lifting characters. With our choices, the first case above follows from statement 1.b in [\emph{loc.\,cit.}]. In the second case, we have that $\Theta_{V_1, W, \gamma, \psi}(\mu \gamma_{\vert K^1}^{-1})=\mu \gamma_{\vert K^1}^{-1}$. Therefore, in this case, it follows from statement 1.c in [\emph{loc.\,cit.}] that $r_{U'}(\Theta_{V, W,\gamma,\psi}(\mu))=\gamma \lvert \  \rvert_{K}^{1/2} \otimes \mu \gamma_{\vert K^1}^{-1}$, where $U'$ denotes the unipotent subgroup of $B'$ and $r_{U'}$ the corresponding Jacquet functor. Thus, an application of (standard) Frobenius reciprocity shows that $\Theta_{V, W,\gamma,\psi}(\mu)$ is a subrepresentation of $i_{B'}^{\U(V)}\left(\gamma \lvert \  \rvert_{K}^{-1/2} \otimes \mu \gamma_{\vert K^1}^{-1}\right)$, whereas an application of the Bernstein form of Frobenius reciprocity shows that it is a quotient of $i_{B'}^{\U(V)}\left(\gamma \lvert \  \rvert_{K}^{1/2} \otimes \mu \gamma_{\vert K^1}^{-1}\right)$, as desired. 
\end{proof}
\begin{remark}
Most of the previous theorem can also be deduced from \cite[\S2.12]{AWS}.
\end{remark}

We will also need a criterion to determine whether $\Theta_{V_1,W, \gamma,\psi}(\mu)$ is nonzero. Fix an element $\delta\in K^\times$ of trace equal to zero. We adapt the sign characters as defined in \eqref{epsilon v} and \eqref{epsilon w} to our case. In particular, we set
\[
\epsilon(W) := \omega_{K/F}(\delta^{-1} \disc(W)) \in \{\pm 1 \}.
\]

\begin{theorem}\label{HarrisKudlaSweet}
	Let $\mu$ be a character of $\U(W)$. The theta lift $\Theta_{V_1, W, \gamma, \psi}(\mu)$ is nonzero if and only if 
	\[
	\epsilon(V_1)\epsilon(W) = \epsilon_{K}\left(\frac{1}{2}, \gamma\tilde{\mu}^{-1}, \psi(\Tr_{K/F}(-\delta(\cdot))) \right).
	\]
	where $\tilde{\mu}=\mu\circ i$ and $i$ denotes the inverse of the isomorphism $K^\times/F^\times\xrightarrow{\simeq} K^1$ defined by $x\mapsto x/x^c$.
\end{theorem}
\begin{proof}
	This is \cite[Proposition 3.4]{R}. Indeed, if $V_1$ is the 1-dimensional Hermitian space with Hermitian form $(x,y)\mapsto 2 xy^c$ and $W$ is the 1-dimensional skew-Hermitian space with skew-Hermitian form $(x,y)\mapsto \delta xy^c$ for an element $\delta\in K^\times$ of trace zero, then it follows from [\emph{loc.\,cit.}] that the theta lift $\theta_{V_1, W, \gamma, \psi}(\mu)$ is nonzero if and only if
	\[
	\gamma(2\delta)^{-1}\mu(-1)\epsilon_K\left(\frac{1}{2},\gamma\tilde{\mu}^{-1},\psi\circ \Tr_{K/F}\right)=1.
	\]
	Now, in this particular case, we have $\epsilon(V_1)=2$, $\epsilon(W)=1$ and
	\begin{align*}
	\epsilon_K\left(\frac{1}{2},\gamma\tilde{\mu}^{-1},\psi(\Tr_{K/F}(-\delta (\cdot)))\right) &=\gamma\tilde{\mu}^{-1}(-\delta)\epsilon_K\left(\frac{1}{2},\gamma\tilde{\mu}^{-1},\psi\circ \Tr_{K/F}\right) \\ &=\gamma(-\delta)\mu(-1)\epsilon_K\left(\frac{1}{2},\gamma\tilde{\mu}^{-1},\psi\circ \Tr_{K/F}\right),
	\end{align*}
	so $\theta_{V_1, W, \gamma, \psi}(\mu)$ is nonzero if and only if
	\[
    \epsilon_K\left(\frac{1}{2},\gamma\tilde{\mu}^{-1},\psi(\Tr_{K/F}(-\delta (\cdot)))\right) =	\gamma(-2\delta^2),
	\]
	and, since $-\delta^2\in \norm_{K/F}(K^\times)$, the right hand side of the last equation becomes $\gamma(2)=\epsilon(V_1)\epsilon(W)$, so we obtain the result in this case. To obtain the result for arbitrary 1-dimensional Hermitian and skew-Hermitian spaces $V_1$ and $W$, we just need to observe that scaling the form on one of these spaces by $a\in F^\times$ amounts to replacing the Weil representation $\Omega_{V,W,\gamma,\psi}$ by $\Omega_{V,W,\gamma,\psi_a}$.
\end{proof}

\subsection{Theta correspondence for $\GL_1\times\GL_3$}\label{subsection: Theta correspondence for GL1 times GL3}

In this subsection, we briefly review the results that we will later need regarding the theta correspondence for the pair $\GL_1\times \GL_3$. This is the case that arises if, in the previous setting, we replace the quadratic field extension $K/F$ by the \'etale quadratic $F$-algebra $F\times F$. This theta correspondence has been described by M\'inguez \cite{minguez}.

Let $F$ be a nonarchimedian local field and fix a nontrivial additive character $\psi: F\rightarrow \bb{C}^\times$. Let $M_{n,m}(F)$ denote the space of $n\times m$ matrices with coefficients in $F$ and let $S_{n,m}$ denote the space of locally constant compactly supported $\bb{C}$-valued functions on $M_{n,m}(F)$. The Weil representation $\omega_\psi$ of the metaplectic group $\mathrm{Mp}_{2nm}(F)$ can be realized on the space $S_{n,m}$. The choice of a character $\gamma$ of $F^\times$ determines a lifting
\[
\tilde{\iota}_{\gamma,\psi}: \GL_n(F)\times \GL_m(F) \longrightarrow \mathrm{Mp}_{2nm}(F)
\]
of the natural map
\[
\iota:\GL_n(F)\times \GL_m(F) \longrightarrow \mathrm{Sp}_{2nm}(F)
\]
defined by
\[
(\tilde{\iota}_{\gamma,\psi}(g,h)f)(x)=\gamma(\det(g))^m \vert \det(g)\vert_F^{m/2} f(g^T x h) \gamma(\det(h))^n \vert \det(h)\vert_F^{n/2}
\]
for all $g\in \GL_n(F)$, $h\in \GL_m(F)$, $x\in M_{n,m}(F)$ and $f\in S_{n,m}$. Observe that our conventions differ from those adopted in \cite{minguez}.

Using the lifting $\tilde{\iota}_{\gamma,\psi}$, we can define the big theta lift $\Theta(\pi)$ of a smooth admissible irreducible representation $\pi$ of $\GL_m(F)$ and, if it is nonzero, the small theta lift $\theta(\pi)$.

Specialize now to the case $m=1$ and $n=3$. The main result in \cite{minguez} is as follows.

\begin{theorem}\label{thm:minguez}
Let $\mu$ be a character of $\GL_1(F)$. Then $\Theta(\mu)$ is nonzero and $\theta(\mu)$ is the Langlands quotient $\pi(\gamma\lvert \ \rvert_F^{1/2}, \mu\gamma^{-2}, \gamma \lvert \ \rvert_F^{-1/2})$ of $i_{B'}^{\GL_3}(\gamma\lvert \ \rvert_F^{1/2} \otimes \mu\gamma^{-2} \otimes \gamma\lvert \ \rvert_F^{-1/2})$, where $B'$ denotes the standard Borel subgroup of $\GL_3$.
\end{theorem}

\section{Theta correspondence for the pair $(\PU_3\rtimes \bb{Z}/2\bb{Z}, G_2)$}\label{section: Theta correspondence PU3 times G2}
In this section, we collect general results on the theta correspondence for $(\PU_3 \rtimes \mathbb{Z}/2\mathbb{Z}, G_2)$. In Section~\ref{group g2 parabolics}, we recall the basic facts for the group $G_2$ and its parabolic subgroups that play a role in this paper. In Section~\ref{subsection: the group PU3 times gal(K/F)}, we do the same for the group $\PU_3 \rtimes \Gal(K/F)$. In Section~\ref{subsection: dual pair}, we discuss the general theory of theta correspondence for $(\PU_3 \rtimes \mathbb{Z}/2\mathbb{Z}, G_2)$. This draws from theta lifting for $(\PU_3, G_2)$, which we discuss in Section~\ref{subsection: results bakic savin} following \cite{BS} closely. Finally, we do the same for the pair $(\PGL_3, G_2)$ in Section~\ref{subsec: theta pgl3 to g2}, which we use later for the split places.

\subsection{The group $G_2$}\label{group g2 parabolics} Let $F$ be a nonarchimedean local field of characteristic zero, with normalized absolute value denoted by $\lvert \ \rvert$. Let $G$ be a split exceptional group of type $G_2$. Let  $B = TU$ be a Borel subgroup of $G$ whose Levi component $T$ is a split maximal torus of $G$. Denote by 
\[
\{ \alpha, \beta, \alpha + \beta, 2 \alpha + \beta, 3 \alpha + \beta, 3 \alpha + 2 \beta  \}
\] 
the corresponding set of positive roots, where $\alpha$ is the short simple root and $\beta$ is the long simple root. For any root $\gamma$ we will denote by $w_\gamma$ the corresponding reflection on the Weyl group of $G$. We can identify 
\[
T \simeq F^\times \times F^\times , \ t \mapsto \left( (2\alpha + \beta)(t), (\alpha + \beta)(t) \right). 
\]
Hence, under this identification, if $(t_1, t_2) \in F^\times \times F^\times$, we have $\alpha(t_1, t_2) = t_1 t_2^{-1}$ and $\beta(t_1, t_2) = t_2^2 t_1^{-1}$. 

Let $Q_1$ be the parabolic subgroup corresponding to the root $\beta$ with Levi decomposition $Q_1 = L_1U_1$. We have an isomorphism $L_1 \simeq \GL_2$ under the map determined by 
\begin{equation}\label{isoL1GL2}
t \mapsto \diag\left((\alpha + \beta)(t), \alpha(t) \right). 
\end{equation}
The group $Q_1$ is usually called the three-step parabolic of $G_2$, as $U_1$ admits a three-step filtration $U_1 = U_1(1) \supset U_1(2) \supset U_1(3) \supset U_1(4) = 1$ given in \cite[Section 3.1]{BS}. Similarly, let $Q_2$ be the parabolic corresponding to the root $\alpha$. Consider its Levi decomposition $Q_2 = L_2 U_2$. We have an isomorphism $L_2 \simeq \GL_2$ under the map determined by 
\begin{equation}\label{isloL2GL2}
t \mapsto \diag\left((2\alpha + \beta)(t), (\alpha + \beta)(t) \right). 
\end{equation}
The group $Q_2$ is usually called Heisenberg parabolic of $G_2$. For $i \in \{ 1,2\}$, We denote by $\delta_{Q_i}$ the modular character of $Q_i$.

\subsection{The group $\PU_3 \rtimes \Gal(K/F)$}\label{subsection: the group PU3 times gal(K/F)} Let $K$ be a quadratic extension of $F$, and let $\U_3$ be the unitary group defined in the previous section, which we identify with matrices in $\GL_3(K)$ fixing the Hermitian form determined by 
\[
\Phi_3 = e = \begin{pmatrix}
0 & 0 & -1 \\
0 & -1 & 0 \\
-1 & 0 & 0 
\end{pmatrix}.
\]
Let $\PU_3$ be the quotient of $\U_3$ by its center. Equivalently, the group $\PU_3$ is the group of $F$-points of the group scheme $\underline{\U}_3/\underline{\U}_1$. Denote by $B'$ the standard Borel subgroup of $\PU_3$ consisting of upper triangular matrices with Levi decomposition $B'=T'U'$, where $T'$ is the standard maximal torus of $\PU_3$. We have \[T' \simeq K^\times\] via the map
\[
\diag(a, b, c) \mapsto \frac{a}{b}.
\]
Note that the Galois group $\Gal(K/F)$ acts on $\PU_3$ by acting on its coefficients. We can then consider the semidirect product $\PU_3 \rtimes \Gal(K/F)$ and the subgroups $B' \rtimes \Gal(K/F)$ and $T' \rtimes \Gal(K/F)$.

\subsection{The dual pair $(G_2, \PU_3 \rtimes \Gal(K/F))$ and exceptional correspondence}\label{subsection: dual pair}
We closely follow \cite[\S1]{BS}. Let $\bb{O}$ be an octonion algebra over $F$ and let $J = J_3(K)$ be the set of $3\times 3$ Hermitian matrices with coefficients in $K$. We can define a structure of Jordan algebra over $F$ on $J$ as follows. Addition on $J$ is given by addition of matrices and regarding multiplication, let $e$ be as above and define $x \circ y = \frac{1}{2}(x e^{\#} y + y e^{\#} x)$, where the superindex $\#$ denotes taking the adjoint, so that $xx^{\#}=\det(x)$. Note that from the expression of $e$ given above, we have $e^\# = e$. The algebra $J$ is equipped with an anti-involution given by the action of the nontrivial element of $\Gal(K/F)$ on the matrix entries. Let $$G := \mathrm{Aut}(\bb{O}),$$ which is a split exceptional group of type $G_2$. On the other hand, let $$G' := \mathrm{Aut}(J) \simeq \PU_3 \rtimes \Gal(K/F),$$ where $\Gal(K/F)$ acts on $\PU_3$ by acting on the coefficients. Consider the Lie algebras $\mathfrak{g} = \mathrm{Lie}(G)$ and $\mathfrak{g'} = \mathrm{Lie}(G')$. Then, 
\[
\mathfrak{h} := \mathfrak g \oplus \mathfrak g' \oplus \mathbb{O}^\circ \oplus J^\circ,
\]
 has a structure of a simple exceptional Lie algebra over $F$, where the superscript $\circ$ denotes the trace zero elements. Let $H = \mathrm{Aut}(\mathfrak h)$, which is disconnected quasi-split of absolute type $E_6$. By definition of $H$, we have an inclusion 
\[
G \times G' \subset H
\]
and they form a dual reductive pair. Let $\Pi$ be the minimal representation of $H$ (see \cite[\S 1.5]{BS} for the definition of minimal representation and \cite{GanSavinMR} for its construction). We also use $\Pi$ to denote the restriction of this representation to $G \times G'$, or the restriction to $G \times \PU_3$. This representation induces correspondences between certain representations of $G'$ and $G$, and certain representations of $\PU_3$ and $G$.

\begin{definition}\label{definition theta pu3}
	Let $\tilde{\sigma}$ be a smooth irreducible representation of $G'$. Define $\Theta(\tilde{\sigma})$ to be the smooth representation of $G$ such that  $\Theta(\tilde{\sigma}) \boxtimes \tilde{\sigma}$ is the maximal $\tilde{\sigma}$-isotypic quotient of $\Pi$, viewed as a representation of $G \times G'$. Similarly, let $\sigma$ be a smooth irreducible representation of $\PU_3$. Define $\Theta_{\PU_3}(\sigma)$ to be the smooth representation of $G$ such that  $\Theta_{\PU_3}(\sigma) \boxtimes \sigma$	is the maximal $\sigma$-isotypic quotient of $\Pi$, viewed as a representation of $G \times \PU_3$.  
\end{definition}

\begin{remark}\label{remark: extensions of representations}
Let $\sigma$ be a smooth irreducible representation of $\PU_3$. Then, it can be verified that 
\begin{equation}\label{eq:comparison between theta lift from PU3 and G'}
\Theta_{\PU_3}(\sigma) = \Theta\left(\Ind_{\PU_3}^{\PU_3 \rtimes \Gal(K/F)}(\sigma)\right),
\end{equation}
where the right-hand side is interpreted as we now explain.         

Let $\sigma^c$ be the representation obtained by twisting $\sigma$ by the action of the nontrivial element of $\Gal(K/F)$. We then have the following two cases: 
\begin{enumerate}
    \item Suppose $\sigma \not\simeq \sigma^c$. Then, $\Ind_{\PU_3}^{\PU_3 \rtimes \Gal(K/F)}(\sigma)$ is irreducible and its restriction to $\PU_3$ is isomorphic to $\sigma \oplus \sigma^c$. In this case, the right-hand side of \eqref{eq:comparison between theta lift from PU3 and G'} is the theta lift defined in Definition~\ref{definition theta pu3}, and we have 
    \[
    \Theta_{\PU_3}(\sigma) = \Theta_{\PU_3}(\sigma^c) = \Theta\left(\Ind_{\PU_3}^{\PU_3\rtimes \Gal(K/F)}(\sigma)\right).
    \]
    \item Suppose that $\sigma \simeq \sigma^c$. Then, $\Ind_{\PU_3}^{\PU_3 \rtimes \Gal(K/F)}(\sigma) \simeq \sigma_1 \oplus \sigma_2$ as representations of the group $\PU_3 \rtimes \Gal(K/F)$, where $\sigma_1, \sigma_2$ are two non-isomorphic irreducible representations of $\PU_3\rtimes \Gal(K/F)$ whose restriction to $\PU_3$ is isomorphic to $\sigma$. In this case, the right-hand side of \eqref{eq:comparison between theta lift from PU3 and G'} is defined as $\Theta(\sigma_1)\oplus\Theta(\sigma_2)$, and we have 
    \[
    \Theta_{\PU_3}(\sigma) = \Theta(\sigma_1) \oplus \Theta(\sigma_2).
    \]
\end{enumerate}
\end{remark}

To define a notion of small theta lift of representations of $\PU_3$ and $G'$, we need the following result from \cite{BS}. 

\begin{theorem}
	Let $\tilde{\sigma}$ be an irreducible representation of $G'$. Then, $\Theta(\tilde{\sigma})$ is a representation of finite length of $G$. Similarly, if $\sigma$ is an irreducible representation of $\PU_3$, we have that $\Theta_{\PU_3}(\sigma)$ is a representation of finite length.
\end{theorem}
\begin{proof}
	The result for $\Theta(\tilde{\sigma})$ is given in Theorem~4.1~(iv) of \cite{BS}. More precisely, the statement for nontempered representations is given in \cite[Proposition~4.3]{BS} and it is proven in an analogous way as \cite[Proposition~4.2]{BS}. The statement for tempered representations follows from Proposition~4.5, Proposition~4.8 and Proposition~4.15 of \cite{BS}. Finally, the result for $\Theta_{\PU_3}(\sigma)$ can be deduced from there using the relation between theta lifts from $G'$ to $G$ and theta lifts from $\PU_3$ to $G$ described in Remark~\ref{remark: extensions of representations}.
\end{proof}


\begin{definition}\label{definition small theta pu3}
	Let $\tilde{\sigma}$ be a smooth irreducible representation of $G'$. Define the small theta lift of $\tilde{\sigma}$, denoted by $\theta(\tilde{\sigma})$, to be the maximal semisimple quotient of $\Theta(\tilde{\sigma})$. We similarly define the small theta lift of a smooth irreducible representation $\sigma$ of $\PU_3$, and denote it by $\theta_{\PU_3}(\sigma)$.
\end{definition}

\subsection{Results of Baki\'c--Savin on the nontempered correspondence} \label{subsection: results bakic savin}

We now explain a result of \cite{BS} which determines the exceptional theta lift of nontempered representations. It is described in terms of the following correspondence studied by Roberts in \cite{Roberts}. Let $\tilde{\omega}$ be the representation of $L_1 \times (T' \rtimes \Gal(K/F))$ given (with this same notation) in \cite[Section 3.1]{BS}. We recall that $L_1\simeq \GL_2$ and $T'\simeq K^\times$. The representation $\tilde{\omega}$ defines a big theta lift from $T' \rtimes \Gal(K/F)$ to $L_1$ given as follows.

\begin{lemma}\label{lemma: theta lift induced by tilde omega}
	Let $\tilde{\omega}$ be the representation of $L_1 \times \left(T' \rtimes \Gal(K/F) \right)$ introduced above and consider the lift it induces from $T' \rtimes \Gal(K/F)$ to $L_1$. Let $\chi$ be a character of $K^\times$. Then, 
	\begin{enumerate}
		\item If $\chi = \chi^c$, write $\Ind_{K^\times}^{K^\times \rtimes \Gal(K/F)} = \chi_1 \oplus \chi_2$ as $K^\times \rtimes \Gal(K/F)$ representations. Then, exactly one of these two characters has a nonzero big theta lift by the correspondence induced by $\tilde{\omega}$, which is equal to the automorphic induction of $\chi$ to $\GL_2(F)$, given by the representation $\tau$ introduced in Proposition~\ref{prop: automorphic induction of chi when chic = chi}.
		\item If $\chi \neq \chi^c$, then $\Ind_{K^\times}^{K^\times \rtimes \Gal(K/F)}(\chi)$ is irreducible and its big theta lift by the correspondence induced by $\tilde{\omega}$ is equal to the automorphic induction of $\chi$ to $\GL_2(F)$. 
	\end{enumerate}
\end{lemma}
\begin{proof}
    This result is given in Section 3.4 of \cite{BS}. 
    
\end{proof}

Let $\chi$ be a unitary character of $K^{\times}$ and $s>0$. Consider the character  $\chi \lvert \Norm_{K/F}(\ ) \rvert^s$ of $K^\times \simeq T'$. We may inflate this character to $B'$ with trivial action on $U'$. First set
\begin{align*}
I(\chi,s)\coloneqq i_{B'}^{\PU_3}(\chi \lvert \Norm_{K/F}(\ ) \rvert^s )
\end{align*}
to denote a principal series representation of $\PU_3$. When $\chi \neq \chi^c$, then $\chi_1 \coloneqq \Ind_{K^\times}^{K^\times \rtimes \Gal(K/F)}\chi$ is an irreducible representation of $T' \rtimes \Gal(K/F)$. Now $\chi_1|N_{K/F}(\ )|^s$ is once again inflated to $B' \rtimes \Gal(K/F)$. Set
\begin{equation}\label{induction of chi to G'}
I(\chi_1,s) \coloneqq i_{B' \rtimes \Gal(K/F)}^{G'}\left(\chi_1|N_{K/F}(\ )|^s\right).
\end{equation}
When $\chi=\chi^c$, then
\[
\Ind_{K^\times}^{K^\times \rtimes \Gal(K/F)}\chi = \chi_1 \oplus \chi_2,
\]
where we denote by $\chi_1$ the character with a nontrivial theta lift to $L_1$ via the theta correspondence induced by $\tilde{\omega}$ (see Lemma~\ref{lemma: theta lift induced by tilde omega}). Note that the character $\chi_2$ has then a trivial theta lift via the correspondence induced by $\tilde{\omega}$. We define representations $I(\chi_1,s)$ and $I(\chi_2, s)$ of $G'$ using the same formula as in \eqref{induction of chi to G'}.

\begin{proposition}\label{prop: exceptional theta lift non tempered representation}
	Let $\chi$ be a unitary character of $K^\times$ and let $s > 0$ be a positive real number. 
	\begin{enumerate}
		\item Let $\sigma_1$ be the unique irreducible quotient of $I(\chi_1,s)$ and denote by $\lvert  \det \rvert^{s}\tau$ the representation of $L_1$ obtained as the big theta lift of the character $\chi_1 \lvert \Norm_{K/F}(\ )\rvert^s$ of $T' \rtimes \Gal(K/F)$ via the correspondence given by $\tilde{\omega}$. Then $\Theta(\sigma_1)$ is a nonzero quotient of $i_{Q_1}^G(\lvert  \det \rvert^{s}\tau)$.
		\item If $\chi$ is $\Gal(K/F)$-invariant, let $\sigma_2$ be the unique irreducible quotient of $I(\chi_2, s)$. Then, $\Theta(\sigma_2) = 0$. 
		\item Let $\sigma$ be the unique irreducible quotient of $I(\chi,s)$. Then $\Theta_{\PU_3}(\sigma) = \Theta(\sigma_1)$ is a nonzero quotient of $i_{Q_1}^G(\lvert  \det \rvert^{s}\tau)$.
	\end{enumerate}	
\end{proposition}
\begin{proof}
	The proof that $\Theta(\sigma_1) \neq 0$ is given in \cite[Proposition 4.2]{BS}. The fact that $\Theta(\sigma_1)$ is a quotient of $i_{Q_1}^G(\tau)$ and that $\Theta(\sigma_2) = 0$ is given in \cite[Proposition 4.3]{BS}. The third point can be proved from the first two points and Remark~\ref{remark: extensions of representations} once we make the following observation. If $\chi$ is $\Gal(K/F)$-invariant, then $\Ind_{\PU_3}^{\PU_3 \rtimes \Gal(K/F)}(\sigma) = \sigma_1 \oplus \sigma_2$. On the other hand, if $\chi$ is not $\Gal(K/F)$-invariant, $\Ind_{\PU_3}^{\PU_3 \rtimes \Gal(K/F)}(\sigma) = \sigma_1$. Alternatively, the third point can be justified in an analogous way as the proof of the first point, but working with the restriction of $\Pi$ to $G \times \PU_3$ and using the following fact. The big theta lift of $\chi \lvert \Norm_{K/F}(\ ) \rvert^s $ from $T'$ to $L_1$ induced by the representation $\tilde{\omega}$ is equal to the big theta lift of $\chi_1 \lvert \Norm_{K/F}(\ ) \rvert^s $ from $T'\rtimes \Gal(K/F)$ to $L_1$ induced by the representation $\tilde{\omega}$. 
\end{proof}

\begin{remark}\label{rmk: tau is the automorphic induction of chi}
    Consider the same notation as in Proposition~\ref{prop: exceptional theta lift non tempered representation}. In view of Lemma~\ref{lemma: theta lift induced by tilde omega}, $ \lvert  \det \rvert^{s}\tau$ can also be described as the automorphic induction of the character $\chi \lvert \Norm_{K/F}(\ )\rvert^s$ to $\GL_2(F)$. 
\end{remark}


\subsection{Theta correspondence for $\PGL_3 \times G_2$}\label{subsec: theta pgl3 to g2} Following \cite{endoscopiclifting}, we state the results we will need regarding the exceptional theta correspondence for $\PGL_3 \times G$. Let $H$ be (the $F$-points of) a split adjoint linear algebraic group of type $E_6$ over $F$. Denote by $\Pi$ the minimal representation of $H$ as in \cite[\S5]{endoscopiclifting}. There is a dual reductive pair $\PGL_3 \times G \xhookrightarrow{} H$ and we can consider the restriction of $\Pi$ to $\PGL_3 \times G$. Using the representation $\Pi$, define the big theta lift $\Theta(\pi)$ of an irreducible admissible representation $\pi$ of $\PGL_3$ in a similar way as above. The next proposition is a particular case of \cite[Corollary 9]{endoscopiclifting}. 

\begin{proposition}\label{prop:theta lift pgl3 g2 general}
	Let $\chi_1, \chi_2, \chi_3$ be characters of $F^\times$ satisfying $\chi_1 \chi_2 \chi_3 = 1$. For every $i$ from $1$ to $3$, write $\chi_i = \mu_i \lvert \ \rvert_F^{s_i}$, where $\mu_i$ is unitary and suppose that the characters are ordered so that $s_1 \geq s_2 \geq s_3$. Consider the representation of $L_2\simeq \GL_2$ given by $\tau = \pi(\chi_3^{-1}, \chi_2^{-1})$. Then, $\Theta(\pi(\chi_1, \chi_2, \chi_3))$ is a nonzero quotient of $i_{Q_2}^G(\tau)$.
\end{proposition}

\section{Local dihedral long root A-packets for $G_2$}\label{sec: A-packets}

In this section, we propose a construction of dihedral long root (local) A-packets for the group $G_2$, following Arthur's conjectures. In Section~\ref{subsection: dihedral long root A par}, we describe the corresponding A-parameter of $G_2$ in the global setting. This parameter factors through an A-parameter of $\prescript{L}{}{\PU_3}$. Thus, in Section~\ref{subsec: howe ps a packets}, we discuss Arthur's conjectures for  $\prescript{L}{}{\PU_3}$ in the local nonarchimedean setting, followed by an explicit construction of the corresponding local A-packets, called Howe--PS packets. We use the results on the theta correspondence for unitary groups outlined in Section~\ref{section: Theta correspondence for unitary groups} for this process. This construction was first given by Gelbart and Rogawski in \cite{GR} and it gives a unitary version of the original construction of Howe-PS packets, which was done in the orthogonal-symplectic setting. In Section~\ref{subsection: from pu3 to pu3 z/2z}, we resolve some details regarding the lift of representations from $\PU_3$ to $\PU_3 \rtimes \bb{Z}/2\bb{Z}$. Next, in Section~\ref{subsection: local a packets g2}, we describe Arthur's conjectures for the dihedral long root A-parameters of $G_2$ in the local nonarchimedean setting and propose a construction of the corresponding A-packets. We use the results from Section~\ref{section: Theta correspondence PU3 times G2} for this process.

Throughout this section $F$ denotes a number field and $L_F$ denotes the corresponding conjectural Langlands group. We denote by $G$ the $F$-points of a split group of type $G_2$ over $F$. 

\subsection{Dihedral long root A-parameters}\label{subsection: dihedral long root A par}
Let $K/F$ be a quadratic field extension with Galois group $\Gal(K/F)=\langle c \rangle$. Let $\chi$ be a character of $\bb{A}_K^\times/K^\times$, which we can also regard as a character of the Weil group $W_K$. Let $\tau$ denote the automorphic representation of $\GL_2$ obtained from $\chi$ by automorphic induction, \textit{i.e.}, the automorphic representation with $L$-parameter
\[
\rho_\tau = \Ind_{W_K}^{W_F} \chi.
\]
The representation $\tau$ is called dihedral with respect to the quadratic field extension $K/F$. Its central character is given by $\omega_{K/F}\cdot\chi_{\vert \bb{A}_F^\times}$. Assume that the character $\chi$ is conjugate-symplectic, \textit{i.e.},
\[
\chi_{\vert \bb{A}_F^\times} = \omega_{K/F}.
\]
Under this assumption, the representation $\tau$ has trivial central character and can therefore be regarded as a representation of $\PGL_2$. Alternatively, this assumption implies that $\det(\rho_\tau)=1$, and therefore $\rho_\tau$ defines an $L$-parameter of $\PGL_2$. We will also assume that $\chi\neq\chi^c$, so that $\rho_\tau$ is irreducible and $\tau$ is a cuspidal representation.

The root system of $G$ contains pairs of orthogonal roots, always consisting of a short and a long root. A choice of such a pair determines a commuting pair of $\SL_2$ subgroups inside $G$,
\[
(\SL_{2,s}\times \SL_{2,l})/\mu_2\subset G,
\]
where $\mu_2=\lbrace \pm 1\rbrace$ is embedded diagonally inside $\SL_{2,s}\times \SL_{2,l}$ in the natural way. Here $\SL_{2,s}$ corresponds to the short root and $\SL_{2,l}$ corresponds to the long root. Moreover, the groups $\SL_{2,s}$ and $\SL_{2,l}$ are centralizers of each other inside $G$. Note that the Langlands dual group of $G$ is $G(\bb{C})$ (we abuse notation and denote by $G(\bb{C})$ the $\bb{C}$-points of our fixed split group of type $G_2$ over $F$).

\begin{definition} \label{global A par G2}
	The long root A-parameter of $G$ associated with $\tau$ is given by
	\[
	\psi_{\tau, l}: L_F \times \SL_2(\bb{C}) \twoheadrightarrow{} W_F \times \SL_2(\bb{C}) \xrightarrow{\rho_\tau \times \id} \SL_{2,s}(\bb{C}) \times_{\mu_2} \SL_{2,l}(\bb{C}) \subset G(\bb{C}).
	\]
\end{definition}

There is a subgroup $\SL_{3,l}(\bb{C})\subset G_2(\bb{C})$ corresponding to the root system of type $A_2$ formed by the six long roots of $G$. Let $T_s(\bb{C})$ denote the diagonal torus of $\SL_{2,s}(\bb{C})$. Then $T_s(\bb{C})\times_{\mu_2}\SL_{2,l}(\bb{C})\subset \SL_{3,l}(\bb{C})$. The group $\SL_{3,l}(\bb{C})$ has index two in its normalizer $N_{G(\bb{C})}(\SL_{3,l}(\bb{C}))$ inside $G(\bb{C})$. Let $w$ denote the standard Weyl element of $\SL_2$. Then $(w,w)\in\SL_{2,s}(\bb{C})\times_{\mu_2}\SL_{2,l}(\bb{C})$ belongs to the non-identity component of $N_{G(\bb{C})}(\SL_{3,l}(\bb{C}))$. Therefore,
\[
N_{\SL_{2,s}(\bb{C})}(T_s(\bb{C}))\times_{\mu_2}\SL_{2,l}(\bb{C}) \subset N_{G(\bb{C})}(\SL_{3,l}(\bb{C}))=\SL_{3,l}(\bb{C})\rtimes \bb{Z}/2\bb{Z}\subset G(\bb{C}).
\]
Note that $^L \PU_3 = \SL_{3,l}(\bb{C}) \rtimes \bb{Z}/2\bb{Z}$.  Hence, the parameter $\psi_{\tau, l}$ factors through an A-parameter of $\PU_3$.
\begin{definition}
	We denote by $\psi_{\chi}$ the A-parameter of $\PU_3$ given by
	\[
	\psi_{\chi}: L_F \times \SL_2(\bb{C}) \xrightarrow{\rho_\tau \times \id} N_{\SL_{2,s}(\bb{C})}(T_s(\bb{C}))\times_{\mu_2}\SL_{2,l}(\bb{C}) \subset \SL_{3,l}(\bb{C})\rtimes \bb{Z}/2\bb{Z} = ^L\!\PU_3.
	\]
\end{definition}
Restricted to $L_K\times \SL_2(\bb{C})$, the parameter $\psi_\chi$ yields a 3-dimensional representation given by
\[
\chi^{-2}\oplus \chi\otimes S_2,
\]
where $\chi$ is regarded as a character of $L_K$ via $L_K\twoheadrightarrow{}W_K$ and $S_2$ denotes the standard $2$-dimensional representation of $\SL_2(\bb{C})$. This information fully determines the equivalence class of $\psi_\chi$ given by $\SL_3(\mathbb{C})$-conjugacy.

\subsection{Howe--PS packets for $\PU_3$}\label{subsec: howe ps a packets} We construct the local Howe--PS packets for $\PU_3$ associated to the A-parameter $\psi_\chi$ at nonarchimedean places using the technique of theta lifting. Let $v$ be a nonarchimedean place of $F$ and $F_v$ the completion of $F$ at $v$. Denote by $L_{F_v}$ the group $W_{F_v} \times \SL_2(\bb{C})$. By fixing an embedding $L_{F_v} \xhookrightarrow{} L_F$, we can precompose the map $\psi_\chi$ with the map $L_{F_v} \times \SL_2(\bb{C}) \xhookrightarrow{} L_F \times \SL_2(\bb{C})$ (which is the identity on $\SL_2(\bb{C})$) to obtain the local A-parameter
\[
\psi_{\chi,v}: L_{F_v} \times \SL_2(\bb{C}) \to ^L\!\PU_3.
\] 
Following Arthur's conjectures, the elements of the local A-packet associated with $\psi$ and $v$ are indexed by irreducible representations of the so-called local component group 
\[
S_{\psi_{\chi,v}} = \pi_0\left(Z_{\SL_3(\bb{C})}(\mathrm{Im}(\psi_{\chi,v})) / Z(\SL_3(\bb{C}))  \right) = \begin{cases}
\bb{Z}/2\bb{Z} \ \text{ if } v \text{ is not split in } K,\\
1 \ \text{ if } v \text{ is split in } K,
\end{cases}
\]
where $Z_{\SL_3(\bb{C})}(\mathrm{Im}(\psi_{\chi,v}))$ denotes the centralizer in $\SL_3(\bb{C})\subset ^L\!\PU_3$ of the image of $\psi_{\chi,v}$, $Z(\SL_3(\bb{C}))$ denotes the center of $\SL_3(\bb{C})$, and $\pi_0$ denotes the group of connected components. Hence, the local A-packet of $\PU_3$ associated to $\psi_\chi$ and $v$ has the form
\[
A_{\psi_{\chi, v}} = \begin{cases}
\{\sigma_v^+, \sigma_v^-\}  \text{ if } v \text{ is not split in } K,\\
\{\sigma_v^+\} \text{ if } v \text{ is split in } K.
\end{cases}
\]
Here, we denote by $\sigma_v^+$ the representation corresponding to the trivial character of $S_{\psi_{\chi,v}} \simeq \bb{Z}/2\bb{Z}$, and in the case that $v$ does not split in $K$, we denote by $\sigma_v^{-}$ the representation of $\PU_3$ corresponding to the nontrivial character of $S_{\psi_{\chi,v}}$. Moreover, Arthur's conjectures predict that for all but finitely many (nonarchimedean) places, the representation $\sigma_v^{+}$ is the unramified representation whose Satake parameter is
\begin{equation}\label{eq: Satake parameter PU_3}
s_{\psi_{\chi,v}}=\psi_{\chi,v}\left(\Phi_v\times \begin{pmatrix} q_v^{-1/2} & 0 \\ 0 & q_v^{1/2} \end{pmatrix}\right), 
\end{equation}
where $\Phi_v$ denotes a geometric Frobenius at $v$ and $q_v$ denotes the cardinality of the corresponding residue field.
Fix a nontrivial additive character $\psi: \bb{A}_F/F\rightarrow \bb{C}^\times$. For each prime $v$ of $F$, we denote by $\psi_v$ the corresponding (nontrivial) additive character of $F_v$. To define the representations of $\PU_3$ appearing in $A_{\psi_{\chi,v}}$ for a nonarchimedian prime $v$ of $F$, we distinguish two cases, according to whether $v$ splits in $K$ or not.

\subsubsection{Nonsplit case}\label{nonsplit A-packet Pu3} Suppose that $v$ is not split in $K$. Denote by $K_v$ the completion of $K$ at the unique place above $v$. Fix a $3$-dimensional Hermitian space $V$ over $K_v$ and identify $\U(V)\simeq \U_3$. Let $V_1$ denote the $1$-dimensional Hermitian space in the Witt tower of $V$. Fix a trace-zero element $\delta \in K_v^\times$ and recall the definition of sign of a skew-Hermitian space given in Section~\ref{section: Theta correspondence for unitary groups} (which depends on the choice of $\delta$). Finally, let $W^+$ and $W^-$ be $1$-dimensional skew-Hermitian spaces over $K_v$ satisfying 
\begin{equation}\label{eq: formula for epsilon(v)epsilon(W)}
\epsilon(V_1)\epsilon(W^+) = \epsilon_K\left(\frac{1}{2}, \chi_v^3, \psi_v\left(\Tr_{K_v/F_v}(- \delta (\cdot))\right) \right) = -\epsilon(V_1)\epsilon(W^-).
\end{equation} 

Let $\gamma = \chi_v$, which is a conjugate symplectic character. Following Section~\ref{subsection: Theta correspondence for U1 times U3}, we use this data to define theta correspondences between the unitary groups defined above. In particular, and following the notation of Section~\ref{subsection: Theta correspondence for U1 times U3}, we can consider the representations $\Omega_{V, W^+, \gamma, \psi_v}$ of $\U(V) \times \U(W^+)$ , and $\Omega_{V, W^-, \gamma, \psi_v}$ of $\U(V)\times \U(W^-)$ and the corresponding big theta lifts $\Theta_{V, W^+, \gamma, \psi_v}$, $\Theta_{V, W^-, \gamma, \psi_v}$.

\begin{definition}\label{big theta lift howe ps}
	We define the following representations of $\U(V) = \U_3$.
	\begin{enumerate}
		\item Denote by $\bb{1}$ the trivial representation of $\U(W^+)$. Define $\sigma_v^+ = \Theta_{V, W^+, \gamma , \psi_v}(\bb{1})$. 
		\item Denote by $\bb{1}$ the trivial representation of $\U(W^-)$. Define $\sigma_v^{-} = \Theta_{V, W^-, \gamma, \psi_v}(\bb{1})$.
	\end{enumerate}
	By the choice of our lifting characters, the center of $\U(V)$ acts trivially on these representations. Hence, we also view $\sigma_v^+$ and $\sigma_v^-$ as representations of $\PU_3$.
\end{definition}

\begin{proposition}\label{prop: simga+ PU3}
    The representation $\sigma_v^+$ is the unique nonzero irreducible quotient of the representation $i_{\bar{B}'}^{\PU_3}(\chi_v \lvert \ \rvert_{K_v}^{1/2}) = I(\chi_v, 1/2)$. In particular, if $\chi_v$ is unramified, the Satake parameter of $\sigma_v^+$ is given by \eqref{eq: Satake parameter PU_3}.
\end{proposition}
\begin{proof}
    Consider the representation $\Omega_{V_1, W^+, \gamma, \psi_v}$ of $\U(V_1) \times \U(W^+)$, which induces a theta lift $\Theta_{V_1, W^+, \gamma, \psi_v}$. By \eqref{eq: formula for epsilon(v)epsilon(W)} and Theorem~\ref{HarrisKudlaSweet}, we have that $\Theta_{V_1,W^+, \gamma, \psi}(\gamma_{\vert K^1}^{-1}) \neq 0$. Therefore, it follows from Proposition~\ref{thetaliftinWitttower} (2) that, as a representation of $\U_3$, $\sigma_v^+$ is an irreducible quotient of 
    \[
    i_{B'}^{\U(V)}\left(\gamma \lvert \ \rvert_{K_v}^{1/2} \otimes \gamma_{\vert K_v^1}^{-1}  \right).
    \]
    This representation is invariant under the center of $\U(V)$, and can therefore be viewed as the representation of $\PU_3$ given by $i_{B'}^{\PU_3}(\chi_v \lvert \ \rvert_{K_v}^{1/2})$ (note that we are denoting the Borel of upper-triangular matrices in $\U_3$ and in $\PU_3$ with the same symbol $B'$). The rest of the assertions follow from there.	
\end{proof}

\begin{proposition}
	The representation $\sigma_v^-$ is nonzero irreducible and supercuspidal.
\end{proposition}
\begin{proof}
    The proof is similar to the proof of the previous proposition. Consider the representation $\Omega_{V_1, W^-, \gamma, \psi_v}$ of $\U(V_1) \times \U(W^-)$, which induces a theta lift $\Theta_{V_1, W^-, \gamma, \psi_v}$. In that case, \eqref{eq: formula for epsilon(v)epsilon(W)} and Theorem~\ref{HarrisKudlaSweet} imply that $\Theta_{V_1, W^-, \gamma, \psi_v}(\gamma_{\vert K^1}^{-1}) = 0$ and the result follows from Proposition~\ref{thetaliftinWitttower} (1).
\end{proof}

\subsubsection{Split case} \label{subsubsection: split howe ps} Suppose that $v$ is split in $K$. Choose a place $w$ of $K$ above $v$, which determines an isomorphism $\U_3(F_v) \simeq \GL_3(F_v)$.
Let $\gamma = \chi_w$. Using this data, we define a theta correspondence for the dual pair $\GL_1(F_v) \times \GL_3(F_v) \xhookrightarrow{} \mathrm{Sp_6}(F_v)$ following Section~\ref{subsection: Theta correspondence for GL1 times GL3}. Denote by $\Theta$ the theta lift defined by this correspondence. 

\begin{definition}
    Let $\bb{1}$ be the trivial representation of $\GL_1(F_v)$. Define the representation $\sigma_v^+ = \theta(\bb{1})$. It is a representation of $\U_3(F_v) \simeq \GL_3(F_v)$ where the center of $\U_3(F_v)$ acts trivially, and therefore we view $\sigma_v^+$ as a representation of $\PU_3(F_v) \simeq \PGL_3(F_v)$. 
\end{definition}

Note that if $\chi$ is unramified at $v$, then it follows from Theorem~\ref{thm:minguez} that $\sigma_v^+$ is the unramified representation with Satake parameter $s_{\psi_{\chi,v}}$. In particular, this holds for all but finitely many split primes $v$, as predicted by Arthur's conjecture. 

\subsection{From representations of $\PU_3$ to representations of $\PU_3 \rtimes \bb{Z}/2\bb{Z}$}\label{subsection: from pu3 to pu3 z/2z} Consider the same notation as in Section \ref{subsec: howe ps a packets}. Suppose that $v$ is not split in $K$, let $W^+$, $W^-$ be the skew-Hermitian spaces introduced in Section \ref{nonsplit A-packet Pu3} and let $\sigma_v^\pm$ be the representations of $\PU_3$ defined in Definition \ref{big theta lift howe ps}. Following Remark~\ref{remark: extensions of representations}, we are interested in knowing whether the induction of the representations $\sigma_v^\pm$ to the group $G' = \PU_3\rtimes \Gal(K_v/F_v)$ is irreducible or decomposes as a direct sum of two non-isomorphic representations. While this question was already studied in Section \ref{subsection: results bakic savin} for the case of principal series, the discussion below will be especially relevant to study the irreducibility of the lift of $\sigma_v^-$ from $\PU_3$ to $G$ via the exceptional theta correspondence.

Recall the following notation introduced in Section \ref{groupsandreps}: if $W$ is a skew-Hermitian space equipped with a skew-Hermitian form $\langle \ , \ \rangle$, we denote $W^{-1}$ the skew Hermitian space with the same underlying space $W$ and skew-Hermitian form given by $-\langle \ , \ \rangle$. In particular, in this section we will consider the spaces $(W^+)^{-1}$ and $(W^-)^{-1}$.

\begin{proposition}\label{prop: irreducibility of induction to G'}
    Assume that $\chi_v^2\neq 1$. Then, the representations
    \[
    \Ind_{\PU_3}^{G'} \sigma_v^+ \quad\text{and}\quad \Ind_{\PU_3}^{G'} \sigma_v^-
    \]
    are irreducible.
\end{proposition}
\begin{proof}
    Following Remark~\ref{remark: extensions of representations}, we need to show that, for $?\in\lbrace +,-\rbrace$, we have that $\sigma_v^?\not\simeq (\sigma_v^?)^c$. However, it follows from \cite[Lemma~2.1(ii)]{HKS} that $(\sigma_v^?)^c\simeq (\sigma_v^?)^\vee$, where $(\sigma_v^?)^\vee$ denotes the contragredient of $\sigma_v^?$. Therefore, we need to prove that $\sigma_v^?\not\simeq (\sigma_v^?)^\vee$.

    Recall that we take $\gamma=\chi_v$. The representation $\sigma_v^?=\Theta_{V,W^?,\gamma,\psi_v}(\bb{1}_{W^?})$ is the isotypic component of the Weil representation
    \[
    \Omega_{V,W^?,\gamma,\psi_v}
    \]
    where the center of $\U(V)$ acts trivially. According to \cite[\S2.5]{AWS}, the dual of this Weil representation is
    \[
    \overline{\Omega_{V,W^?,\gamma,\psi_v}}\simeq \Omega_{V,W^?,\gamma^{-1},\bar{\psi}_v}\simeq
    \Omega_{V,(W^?)^{-1},\gamma^{-1},\psi_v},
    \]
    and thus $(\sigma_v^?)^\vee$ is the isotypic component $\Omega_{V,(W^?)^{-1},\gamma^{-1},\psi_v}$ where the center of $\U(V)$ acts trivially. Equivalently, we have that $(\sigma_v^?)^\vee=\Theta_{V,(W^?)^{-1},\gamma^{-1},\psi_v}(\mathbb{1}_{(W^?)^{-1}})$.
    
    According to \cite[Theorem~4.3(4)]{AG}, respectively \cite[Theorem~4.5(1)]{AG}, and using the notation in [\emph{loc.\,cit.}], the $L$-parameter associated with the representation $\sigma_v^+$, respectively $\sigma_v^-$, is $\gamma^{-2}\oplus\gamma\vert\cdot\vert_{K_v}^{1/2}\oplus\gamma\vert\cdot\vert_{K_v}^{-1/2}$, respectively $\gamma^{-2}\oplus\gamma S_2$. On account of the previous paragraph, and using again [\emph{loc.\,cit.}], for the representations $(\sigma_v^+)^\vee$ and $(\sigma_v^-)^\vee$ the corresponding $L$-parameters are $\gamma^{2}\oplus\gamma^{-1}\vert\cdot\vert_{K_v}^{1/2}\oplus\gamma^{-1}\vert\cdot\vert_{K_v}^{-1/2}$ and $\gamma^{2}\oplus\gamma^{-1} S_2$. Since $\gamma\neq\gamma^{-1}$, we conclude that $\sigma_v^?$ and $(\sigma_v^?)^\vee$ belong to different $L$-packets and are therefore non-isomorphic. We note that, although the results in \cite{AG} rely on the local Langlands correspondence, this is known for the group $\U_3$ by the work of Rogawski \cite{Rog90}.
\end{proof}

\begin{proposition}\label{prop: reducibility of induction to G'}
    Assume that $\chi_v^2=1$. Then, the representations
    \[
    \Ind_{\PU_3}^{G'} \sigma_v^+ \quad\text{and}\quad \Ind_{\PU_3}^{G'} \sigma_v^-
    \]
    are not irreducible.
\end{proposition}
\begin{proof}
    Following Remark~\ref{remark: extensions of representations}, we need to show that, for $?\in\lbrace +,-\rbrace$, we have that $\sigma_v^?\simeq (\sigma_v^?)^c$. However, it follows from \cite[Lemma~2.1(ii)]{HKS} that $(\sigma_v^?)^c\simeq (\sigma_v^?)^\vee$, where $(\sigma_v^?)^\vee$ denotes the contragredient of $\sigma_v^?$. Therefore, we need to prove that $\sigma_v^?\simeq (\sigma_v^?)^\vee$.

    Recall that we take $\gamma=\chi_v$. The representation $\sigma_v^?=\Theta_{V,W^?,\gamma,\psi_v}(\bb{1}_{W^?})$ is the isotypic component of the Weil representation
    \[
    \Omega_{V,W^?,\gamma,\psi_v}
    \]
    where the center of $\U(V)$ acts trivially. According to \cite[\S2.5]{AWS}, the dual of this Weil representation is
    \[
    \overline{\Omega_{V,W^?,\gamma,\psi_v}}\simeq \Omega_{V,W^?,\gamma^{-1},\bar{\psi}_v}\simeq
    \Omega_{V,(W^?)^{-1},\gamma,\psi_v},
    \]
    and thus $(\sigma_v^?)^\vee$ is the isotypic component $\Omega_{V,(W^?)^{-1},\gamma,\psi_v}$ where the center of $\U(V)$ acts trivially. 
    Note that we used that $\gamma^{-1} = \gamma$, which holds because $\chi_v^2=1$. Also, the condition $\chi_v^2=1$ implies that $\chi_v=\chi_v^c$. By Hilbert's Theorem~90, we can write $-1=a/a^c$ for some $a\in K^\times$. Therefore, we deduce that $\omega_{K/F}(-1)=\chi_v(-1)=\chi_v(a)/\chi_v^c(a)=1$. In particular, it follows that $-1\in \norm_{K/F}(K^\times)$ and hence $(W^?)^{-1}\simeq W^?$, which shows that $\sigma_v^?\simeq (\sigma_v^?)^\vee$.
\end{proof}

\subsection{Dihedral long root A-packets for $G_2$}\label{subsection: local a packets g2}
We discuss local nonarchimedian A-packets associated to the dihedral long root A-parameter introduced in Definition~\ref{global A par G2}. As before, we obtain the local A-parameter $\psi_{\tau, l, v}$ at a place $v$, whose component group $S_{\psi_{\tau, l,v}}$ is given by
\[
S_{\psi_{\tau, l,v}} = 
\begin{cases}
\bb{Z}/2\bb{Z} & \text{ if } \rho_{\tau,v} \text{ is irreducible (or, $\tau_v$ is a discrete series representation)},\\
1 & \text{ if } \rho_{\tau,v} \text{ is reducible (or, $\tau_v$ is not a discrete series representation)}.
\end{cases} 
\]

\subsubsection{Nonsplit case}
Let $v$ be a nonarchimedian place of $F$ which does not split in $K$, and denote by $K_v$ the completion of $K$ at the unique place above $v$. We continue to use $c$ to denote the generator of $\Gal(K_v/F_v)$. In this case, $\tau_v$ is the smooth irreducible representation of $\PGL_2$ with $L$-parameter $\rho_{\tau,v}=\Ind_{W_{F_v}}^{W_{K_v}}\chi_v$. The representation $\Ind_{W_{F_v}}^{W_{K_v}}\chi_v$ is irreducible if and only if $\chi_v \neq \chi_v^c$. Thus we have the equivalent characterization for a nonsplit $v$:
\[
S_{\psi_{\tau, l,v}} = 
\begin{cases}
\bb{Z}/2\bb{Z} & \text{ if } \chi_v \neq \chi_v^c ,\\
1 & \text{ if } \chi_v = \chi_v^c.
\end{cases}
\]
Arthur's conjectures predict that the local A-packets are as follows:
\begin{equation}\label{A packet G2}
A_{\psi_{\tau, l,v}} = 
\begin{cases}
\{\pi_v^+, \pi_v^-\} & \text{ if } \chi_v \neq \chi_v^c ,\\
\{\pi_v^+\} & \text{ if } \chi_v = \chi_v^c.
\end{cases}
\end{equation}
Here, $\pi_v^+$ is indexed by the trivial character of $S_{\psi,l,v}$ and $\pi_v^-$ by the nontrivial character of $\bb{Z}/2\bb{Z}$ when it occurs.  Arthur's conjectures also predict that, for all but finitely many places, $\pi^+_v$ is the unramified principal series representation whose Satake parameter is given by 
\begin{equation}\label{satake par for A par of G2}
    s_{\psi_{\tau, l ,v}}=\psi_{\tau, l ,v}\left(\Phi_v\times \begin{pmatrix} q_v^{-1/2} & 0 \\ 0 & q_v^{1/2} \end{pmatrix}\right), 
\end{equation}
where, as before, $\Phi_v$ denotes a geometric Frobenius at $v$ and $q_v$ denotes the cardinality of the corresponding residue field.

We carry out the construction of the expected local A-packets using the exceptional theta correspondence for the pair $(G_2, \PU_3 \rtimes \bb{Z}/2\bb{Z})$ introduced in Section~\ref{subsection: dual pair}. Recall that Definition~\ref{big theta lift howe ps} gives us the construction of the representations $\sigma_v^+$ and $\sigma_v^-$ of $\PU_3$.

\begin{definition}\label{definition theta lift g2}
    With the notation $\theta_{\PU_3}$ established in Definition~\ref{definition small theta pu3} in mind, we define
    \begin{enumerate}
        \item $\pi_v^+ \coloneqq \theta_{\PU_3}(\sigma_v^+)$,
        \item $\pi_v^- \coloneqq \theta_{\PU_3}(\sigma_v^-)$.
    \end{enumerate}
\end{definition}
Observe that the local A-packet for $\PU_3$ as constructed in Section~\ref{nonsplit A-packet Pu3} always has two elements when $v$ is nonsplit in $K$, but the local A-packet for $G_2$ as in \eqref{A packet G2} is expected to have two elements when $\chi_v \neq \chi_v^c$ and one otherwise. We construct the latter A-packet by lifting the two representations $\{\sigma_v^+,\sigma_v^-\}$. We show in subsequent sections that there is no discrepancy in the size of the packets; indeed we will see that, while $\pi_v^+$ is always nonzero, $\pi_v^-$ vanishes if and only if $\chi_v=\chi_v^c$. This is summarized in the following theorem.


\begin{theorem}\label{theorem: pi+ for g2} The lifting of the local A-packet $\{\sigma_v^+,\sigma_v^-\}$ for a nonsplit place $v$ via the exceptional theta correspondence for the dual pair $(G_2, \PU_3 \rtimes \Gal(K_v/F_v))$ yields the following:
\begin{enumerate}
    \item $\pi_v^+$ is nonzero and it is equal to the unique irreducible quotient of the parabolically induced representation $i_{Q_1}^G(\lvert\det\rvert^{1/2}\tau_v)$. For all but finitely many nonsplit places $v$, the representation $\pi_v^+$ is the unramified principal series with Satake parameter given by \eqref{satake par for A par of G2}. 
    \item If $\chi_v \neq \chi_v^c$, the representation $\pi_v^-$ is nonzero, irreducible and tempered.
    \item If $\chi_v = \chi_v^c$, the representation $\pi_v^-$ is zero.
\end{enumerate}
\end{theorem}
\begin{proof} 
We prove the first point of the theorem and outline the proofs for the second and third points, which will be completed in the next sections.
\begin{enumerate}
    \item According to Proposition~\ref{prop: simga+ PU3}, the representation $\sigma_v^+$ is the unique irreducible quotient of $i_{B'}^{\PU_3}(\chi_v \lvert \ \rvert_{K_v}^{1/2}) = I(\chi_v, 1/2)$. Now note that $\chi_v$ is unitary (because $\chi$ is conjugate symplectic) and that the automorphic induction of $\chi_v\lvert \ \rvert_{K_v}^{1/2}$ is equal to $\lvert\det\rvert^{1/2}\tau_v$. It therefore follows from Proposition~\ref{prop: exceptional theta lift non tempered representation} and Remark~\ref{rmk: tau is the automorphic induction of chi} that $\Theta_{\PU_3}(\sigma_v^+)$ is a nonzero quotient of $i_{Q_1}^G(\lvert\det\rvert^{1/2}\tau_v)$. Since $\tau_v$ is a tempered representation, we deduce that $i_{Q_1}^G(|\det|^{1/2}\tau_v)$ has a unique irreducible quotient (see \cite[p.~468]{M}), which yields the desired description for $\pi_v^+ = \theta_{\PU_3}(\sigma^+)$.
    \item  If $\chi_v\neq\chi_v^c$, it follows from Theorem~\ref{thm: nonvanishing of pi} that $\pi_v^-$ is not zero, and it follows from Proposition~\ref{prop: generic vanishing} that it does not have any generic subquotient. Therefore, the fact that $\pi^-$ is irreducibile and tempered follows from Proposition~\ref{prop: irreducibility of induction to G'}, \cite[Remark~4.4]{BS} and \cite[Proposition~4.15]{BS}, as we explain in Theorem \ref{thm: pi- is irreducible if chi^2 neq 1} and in its proof.
    \item If $\chi_v=\chi_v^c$, the vanishing of $\pi_v^-$ follows from Theorem~\ref{thm: vanishing}. 
\end{enumerate}
\end{proof}

\subsubsection{Split case} Let $v$ be a nonarchimedian place of $F$ which splits in $K$ and let $w$ and $w'$ be the places of $K$ lying above $v$. There are natural identifications $F_v\simeq K_w\simeq K_{w'}$. Since $\chi$ is conjugate-symplectic, we have that $\chi_{w'}=\chi_w^{-1}$ and $\chi_w$ is unitary. Therefore,
\[
\rho_{\tau,v}=\chi_w\oplus \chi_{w'}=\chi_w\oplus\chi_w^{-1}
\]
and $\tau_v=\pi(\chi_w,\chi_w^{-1})$ is equal to the normalized parabolic induction $i_P^{\GL_2(F_v)}(\chi_w, \chi_w^{-1})$, where here we are using $P$ to denote the Borel of $\GL_2(F_v)$ consisting on upper triangular matrices. Since $\rho_{\tau,v}$ is reducible, it follows that the component group $S_{\psi_{\tau,l,v}}$ is trivial.

The choice of the prime $w$ determines an isomorphism $\PU_3(F_v)\simeq\PGL_3(F_v)$. We construct the predicted singleton A-packet by using theta lifting from $\PGL_3(F_v)$ to $G_2(F_v)$. Recall from Section~\ref{subsubsection: split howe ps} that $\sigma_v^+$ is the theta lift of the trivial character $\mathbbm{1}$ of $F_v^{\times}$ to $\GL_3(F_v^{\times})$ using theta correspondence for the dual pair $(\GL_1, \GL_3)$ and we interpret the result, which has trivial central character, as a representation of $\PGL_3$. More precisely, applying Theorem~\ref{thm:minguez}, with $\mu=1$ and $\gamma=\chi_w$,
\begin{equation*}
    \sigma_v^+ = \pi(\chi_w|\ |^{1/2}_{F_v}, \chi_w^{-2}, \chi_w|\ |_{F_v}^{-1/2}),
\end{equation*} 
the Langlands quotient of $i_{B'}^{\GL_3}(\chi_w|\ |_{F_v}^{1/2}\otimes \chi_w^{-2}\otimes \chi_w|\ |_{F_v}^{-1/2})$, where $B'$ is the standard Borel subgroup of $\GL_3$. Let $\Theta_{\PGL_3}$ stand for big theta lift from $\PGL_3(F_v)$ to $G_2(F_v)$ and let $\tilde{\tau}_v=\pi(\chi_w^{-1}|\ |_{F_v}^{1/2}, \chi_w^2)$, which can be viewed as a representation of $Q_2$. It follows from Proposition~\ref{prop:theta lift pgl3 g2 general} that $\Theta_{\PGL_3}(\sigma_v^+)$ is a nonzero quotient of $i_{Q_2}^{G_2}(\tilde{\tau}_v)$, and therefore that it has finite length. 

\begin{definition}\label{definition theta lift g2 split}
    Define
    \[
    \pi_v^+ \coloneqq \theta_{\PGL_3}(\sigma_v^+)
    \]
    to be the maximal semisimple quotient of $\Theta_{\PGL_3}(\sigma_v^+)$.
\end{definition}

\begin{theorem}\label{split theta lift to g2}
    The representation $\pi_v^+$ is the unique irreducible quotient of $i_{Q_1}^{G_2}(\vert\det\vert^{1/2}\tau_v)$. For all but finitely many nonsplit places $v$, the representation $\pi_v^+$ is the unramified principal series with Satake parameter given by \eqref{satake par for A par of G2}.
\end{theorem}
\begin{proof}
    Let $\tilde{\tau}_v=\pi(\chi_w^{-1}|\ |_{F_v}^{1/2}, \chi_w^2) = i_P^{\GL_2(F_v)}(\chi_w^{-1}|\ |_{F_v}^{1/2}, \chi_w^2)$ be as above. Then, it follows from Proposition~\ref{prop:theta lift pgl3 g2 general} that $\Theta_{\PGL_3}(\sigma_v^+)$ is a nonzero quotient of $i_{Q_2}^{G_2}(\tilde{\tau}_v)$. Following a similar procedure as in the proof of Proposition~\ref{From3stepparabolictoHeisenberg} below, it can be seen that
    \[
    i_{Q_2}^{G_2}(\tilde{\tau}_v)\simeq i_{Q_1}^{G_2}(\vert\det\vert^{1/2}\tau_v).
    \]
    Indeed, this isomorphism follows from \eqref{eq: from parabolic to Borel} and \eqref{eq: change of variables alpha and beta} of the proof of Proposition~\ref{From3stepparabolictoHeisenberg} after we replace the character $\mu$ used there by the character $\chi_w$.
    Since the representation on the right hand side has a unique irreducible quotient (see \cite[p.~468]{M}), the result follows. In fact, similarly as in the case of Remark~\ref{remark: pi+ is irreducible} below, it can be proved that $\Theta_{\PGL_3}(\sigma_v^+)$ is equal to the unique irreducible quotient of this representation, given by $i_{Q_2}^{G_2}(\chi_w\circ\det)$.
\end{proof}

\section{Non-vanishing of theta lifts}\label{section: nonvanishing of theta lifts}

Let $K/F$ be a quadratic extension of local fields of characteristic zero and consider the notation introduced in Section~\ref{section: Theta correspondence PU3 times G2}. In this section, we prove the following criterion regarding non-vanishing of theta lifts from $\PU_3$ to $G$: if the contragredient of a representation $\tau$ of $\PU_3$ (viewed as a representation of $\U_3$) has an irreducible quotient with trivial central character when restricted to a suitable two-variable unitary subgroup, then $\Theta_{\PU_3}(\tau) \neq 0$. This criterion is proven by computing Fourier--Jacobi periods: with the notation introduced in Section~\ref{group g2 parabolics}, we show that, for any nontrivial character $\psi'$ of $U_1(2)/U_1(3)$, 
\[
\Hom_{\PU_3}(\Pi_{U_1(2), \psi'} , \tau) \neq 0.
\]
The key to prove this result is a description of $\Pi_{U_1(2), \psi'}$ in terms of Weil representations attached to two-variable unitary subgroups and $\SL_2$. In Section~\ref{subsec: the groups U(V2) and U(V2')}, we provide a suitable description of two-variable unitary subgroups; in Section~\ref{subsec: an orbit problem}, we introduce the symplectic spaces that determine the aforementioned Weil representations, and, in Section~\ref{subsec: non vanishing of theta lifts}, we prove the criterion.

Fix a nontrivial additive character $\psi:F\rightarrow \bb{C}^\times$ and fix also a conjugate-symplectic character $\chi:K^\times\rightarrow \bb{C}^\times$. Let $\sigma^+$ and $\sigma^-$ be the representations of $\PU_3$ introduced in Definition~\ref{big theta lift howe ps} with respect to these data, and let $\pi^+$ and $\pi^-$ be their corresponding small theta lifts to the group $G$, as introduced in Definition~\ref{definition theta lift g2}.
In Section~\ref{subsec: see saw argument}, we use a see-saw argument to verify that $\sigma^+$, and, under the condition that $\chi^2 \neq 1$, also $\sigma^-$, satisfy the hypothesis of the non-vanishing criterion. As a consequence, $\pi^+ \neq 0$, and, if $\chi^2 \neq 1$, also $\pi^- \neq 0$, proving the non-vanishing statements in Theorem~\ref{theorem: pi+ for g2}.
\subsection{The groups $\U(V_2)$ and $\U(V_2')$}\label{subsec: the groups U(V2) and U(V2')} 
Let $V_2$ and $V_2'$ be 2-dimensional Hermitian spaces over $K$ such that $V_2$ is split and $V_2'$ is nonsplit.

Assume that, with respect to some $K$-basis, the Hermitian form on $V_2$ is defined by the Hermitian matrix $\Phi_2\in H_2(K)$. Then we can describe $\U(V_2)$ in the following way:
\[
\U(V_2)=\lbrace g\in \GL_2(K) \;\colon\; g\Phi_2g^\dagger=\Phi_2\rbrace.
\]
Using this description, the group $\U(V_2)$ acts on the 4-dimensional $F$-space of Hermitian $2\times 2$ matrices $H_2(K)$ by the rule
\[
(g,A)\mapsto gAg^\dagger \quad \text{for }g\in \U(V_2)\text{ and } A\in H_2(K),
\]
and this action preserves the quadratic form defined on $H_2(K)$ by the determinant map. Moreover, it stabilizes the 3-dimensional $F$-subspace
\[
Q=\lbrace A\in H_2(K)\;\colon\; \Tr(A\Phi_2^{-1})=0 \rbrace.
\]
This action defines a homomorphism from $\U(V_2)$ to $\SO(Q)$, the $F$-points of the special orthogonal group associated with the 3-dimensional quadratic $F$-space $Q$. The kernel of this homomorphism is the center $Z(\U(V_2))$ of $\U(V_2)$. This allows us to identify $\PU_2(V_2)=\U(V_2)/Z(\U(V_2))$ with an index-2 subgroup of $\SO(Q)$, which we will denote by $\SO(Q)^+$. Actually, one can show that $\SO(Q)^+$ consists of the elements in $\SO(Q)$ whose spinor norm belongs to $\norm_{K/F}(K^\times)/(F^\times)^2\subset F^\times/(F^\times)^2$. 

Similarly, assume that, with respect to some $K$-basis, the Hermitian form on $V_2'$ is defined by the Hermitian matrix $\Phi_2'\in H_2(K)$. Then we can describe $\U(V_2')$ in the following way:
\[
\U(V_2')=\lbrace g\in \GL_2(K) \;\colon\; g\Phi_2'g^\dagger=\Phi_2'\rbrace.
\]
Using this description, the group $\U(V_2')$ acts on the 4-dimensional $F$-space of Hermitian $2\times 2$ matrices $H_2(K)$ by the rule
\[
(g,A)\mapsto gAg^\dagger \quad \text{for }g\in \U(V_2')\text{ and } A\in H_2(K),
\]
and this action preserves the quadratic form defined on $H_2(K)$ by the determinant map. Moreover, it stabilizes the 3-dimensional $F$-subspace
\[
Q'=\lbrace A\in H_2(K)\;\colon\; \Tr(A\Phi_2'^{-1})=0 \rbrace.
\]
This action defines a homomorphism from $\U(V_2')$ to $\SO(Q')$, the $F$-points of the special orthogonal group associated with the 3-dimensional quadratic $F$-space $Q'$. The kernel of this homomorphism is the center $Z(\U(V_2'))$ of $\U(V_2')$. This allows us to identify $\U(V_2')/Z(\U(V_2'))$ with an index-2 subgroup of $\SO(Q')$, which we will denote by $\SO(Q')^+$. Actually, one can show that $\SO(Q')^+$ consists of the elements in $\SO(Q')$ whose spinor norm belongs to $\norm_{K/F}(K^\times)/(F^\times)^2\subset F^\times/(F^\times)^2$.

\subsection{An orbit problem}\label{subsec: an orbit problem}

We choose the following description of $\U_3$:
\[
\U_3=\lbrace g\in\GL_3(K)\;\colon\; g\Phi_3g^\dagger=\Phi_3\rbrace,
\]
where
\[
\Phi_3 = 
\begin{pmatrix}
0 & 0 & -1 \\
0 & -1 & 0 \\
-1 & 0 & 0
\end{pmatrix},
\]
and we have $\PU_3=\U_3/Z(\U_3)=\U_3/\U_1$. Let $J=H_3(K)$, and
\[
J_1=\lbrace g\in J\;\colon\; \Tr(g\Phi_3^{-1})=1\rbrace.
\]
Let $\Omega_2$ be defined as in \cite[\S2.1]{BS}. Then, as in [\emph{op.\,cit.}, \S3.1.1], we have
\[
\Omega_2\cap J_1=\lbrace g\in J\;\colon\; \Tr(g\Phi_3^{-1})=1 \text{ and } \rank(g)=1\rbrace.
\]
The group $\U_3$ acts on $\Omega_2\cap J_1$ by the rule
\[
(g,A)\mapsto gAg^\dagger \quad \text{for }g\in \U_3\text{ and } A\in \Omega_2\cap J_1,
\]
and this descends to an action of $\PU_3$ on $\Omega_2\cap J_1$. Any matrix $A\in \Omega_2\cap J_1$ can be written as
\[
A=-\lambda u^\dagger u \quad\text{where } u\in M_{1 \times 3}(K), \,\lambda \in F^\times \text{ and } -\lambda u\Phi_3 u^\dagger = 1.
\]
Observe that, if $\lambda\in \norm_{K/F}(K^\times)$, then we can modify $u$ to account for this factor. Fix an element $\lambda_0\in F^\times$ which is not the norm of an element in $K^\times$. Then, we can decompose the set $\Omega_2\cap J_1$ into two subsets:
\begin{align*}
(\Omega_2\cap J_1)_0 &=\lbrace -u^\dagger u \;\colon\; u\in M_{1\times 3}(K),\, -u\Phi_3 u^\dagger=1\rbrace \\
(\Omega_2\cap J_1)_1 &=\lbrace -\lambda_0 u^\dagger u \;\colon\; u\in M_{1\times 3}(K),\, -\lambda_0 u\Phi_3u^\dagger=1\rbrace.
\end{align*}
We claim that these are the orbits for the action of $\PU_3$ on $\Omega_2\cap J_1$. Indeed, this follows from the observation that, if $-u^\dagger u\in (\Omega_2\cap J_1)_0$, then the subspace $\langle u\rangle^\perp$ orthogonal to $u$ with respect to $\Phi_3$ is a split 2-dimensional Hermitian $K$-space, and, if $-\lambda_0 uu^\dagger\in(\Omega_2\cap J_1)_1$, then $\langle u\rangle^\perp$ is a nonsplit 2-dimensional Hermitian $K$-space.

For the orbit $(\Omega_2\cap J_1)_0$, let $u = (0,1,0)$ and choose the element
\[
f_0=-u^\dagger u=\begin{pmatrix} 0 & 0 & 0 \\ 0 & -1 & 0 \\ 0 & 0 & 0 \end{pmatrix}.
\]
Its stabilizer in $\PU_3$ is the subgroup
\begin{equation*}
\left\lbrace \begin{pmatrix} \ast & 0 & \ast \\ 0 & 1 & 0 \\ \ast & 0 & \ast  \end{pmatrix} \in\PU_3 \right\rbrace
\simeq \left\lbrace g\in\GL_2(K)\;\colon\; g\begin{pmatrix} 0 & -1 \\ -1 & 0 \end{pmatrix}g^\dagger = \begin{pmatrix} 0 & -1 \\ -1 & 0 \end{pmatrix}\right\rbrace  \simeq \U(V_2).
\end{equation*}
For the last identification, let $V_2$ be the 2-dimensional $K$-subspace spanned by $u_1=(1,0,0)$ and $u_3=(0,0,1)$, equipped with the Hermitian form given in this basis by
\[
\Phi_2=
\begin{pmatrix}
0 & -1 \\
-1 & 0
\end{pmatrix}.
\]

For the orbit $(\Omega_2\cap J_1)_1$, let $u'=(1,0,1/(2\lambda_0))$ and choose the element
\[
f_1 = -\lambda_0 (u')^\dagger u' = \begin{pmatrix} -\lambda_0 & 0 & -1/2 \\ 0 & 0 & 0 \\ -1/2 & 0 & -1/(4\lambda_0) \end{pmatrix}.
\]
Let $u_1'=u'$, $u_2' = (0,1,0)$ and $u_3'=(1,0,-1/(2\lambda_0))$. Then, with respect to this basis,
\[
f_1=\begin{pmatrix} -\lambda_0 & 0 & 0 \\ 0 & 0 & 0 \\ 0 & 0 & 0 \end{pmatrix}\quad\text{and}\quad \Phi_3=\begin{pmatrix} -1/\lambda_0 & 0 & 0 \\ 0 & -1 & 0 \\ 0 & 0 & 1/\lambda_0 \end{pmatrix},
\]
and therefore the stabilizer of $f_1$ in $\PU_3$ is given by
\begin{equation*}
    \left\lbrace \begin{pmatrix} 1 & 0 & 0 \\ 0 & \ast & \ast \\ 0 & \ast & \ast  \end{pmatrix} \in\PU_3 \right\rbrace  \simeq \left\lbrace g\in\GL_2(K)\;\colon\; g\begin{pmatrix} -1 & 0 \\ 0 & 1/\lambda_0 \end{pmatrix}g^\dagger = \begin{pmatrix} -1 & 0 \\ 0 & 1/\lambda_0 \end{pmatrix}\right\rbrace \simeq \U(V_2').
\end{equation*}
For the last identification, let $V_2'$ be the 2-dimensional $K$-subspace spanned by $u_2'$ and $u_3'$, equipped with the Hermitian form given in this basis by
\[
\Phi_2'=
\begin{pmatrix}
-1 & 0 \\
0 & 1/\lambda_0
\end{pmatrix}.
\]

As in \cite[\S3.1.1]{BS}, to the element $f_0$ we attach the subspace
\[
\Delta_0^\perp = \left\lbrace\begin{pmatrix} a & 0 & y \\ 0 & 0 & 0 \\ \bar{y} & 0 & c \end{pmatrix}\in J\right\rbrace\simeq H_2(K).
\]
This subspace is invariant by the action of the stabilizer of $f_0$ in $\PU_3$, and, under its identification with $\U(V_2)$, the action of the stabilizer of $f_0$ on $\Delta_0^\perp$ agrees with the action described in the previous subsection. It preserves the determinant form on $H_2(K)$ and the decomposition $\Delta_0^\perp\simeq H_2(K)=Q\oplus Z$, with $Q$ defined as above and $Z$ denoting the 1-dimensional subspace of $H_2(K)$ generated by $\Phi_2$. The group $\U(V_2)$ acts trivially on $Z$, whereas, as described above, its action on $Q$ identifies
$\PU(V_2)=\U(V_2)/Z(\U(V_2))$ with an index-2 subgroup of $\SO(Q)$, which we denote by $\SO(Q)^+$.

Similarly, to the element $f_1$ we attach the subspace $\Delta_1^\perp\subseteq J$, which, with respect to the basis $u_1',u_2',u_3'$ defined above, is given by
\[
\Delta_1^\perp = \left\lbrace\begin{pmatrix} 0 & 0 & 0 \\ 0 & a & y \\ 0 & \bar{y} & c \end{pmatrix}\in J\right\rbrace \simeq H_2(K).
\]
This subspace is invariant by the action of the stabilizer of $f_1$ in $\PU_3$, and, under its identification with $\U(V_2')$, the action of the stabilizer of $f_1$ on $\Delta_1^\perp$ agrees with the action described in the previous subsection. It preserves the determinant form on $H_2(K)$ and the decomposition $\Delta_1^\perp\simeq H_2(K)=Q'\oplus Z'$, where $Q'$ is defined as above and $Z'$ denotes the 1-dimensional subspace of $H_2(K)$ generated by $\Phi_2'$. The group $\U(V_2')$ acts trivially on $Z'$, whereas, as described above, its action on $Q'$ identifies
$\PU(V_2')=\U(V_2')/Z(\U(V_2'))$ with an index-2 subgroup of $\SO(Q')$, which we denote by $\SO(Q')^+$.

\subsection{Non-vanishing criterion for theta lifts}\label{subsec: non vanishing of theta lifts}

Let $\Gamma$ denote a 2-dimensional symplectic space. Then $\Gamma\otimes_F \Delta_0^\perp$ and $\Gamma\otimes_F\Delta_1^\perp$ are 8-dimensional symplectic spaces. Let $\omega_0$ and $\omega_1$ be the Weil representations associated with $\Gamma\otimes_F \Delta_0^\perp$ and $\Gamma\otimes_F\Delta_1^\perp$, respectively, for a fixed nontrivial character $\psi':F\rightarrow \bb{C}^\times$. Let $\omega_{0,Q}$, $\omega_{0,Z}$, $\omega_{1,Q'}$ and $\omega_{1,Z'}$ be the Weil representations associated with the symplectic spaces $\Gamma\otimes_F Q$, $\Gamma\otimes_F Z$, $\Gamma\otimes_F Q'$ and $\Gamma\otimes_F Z'$, respectively, for the character $\psi'$. Then, under the natural homomorphism
\[
\Mp(\Gamma\otimes_F Q)\times \Mp(\Gamma\otimes_F Z)\longrightarrow \Mp(\Gamma\otimes_F \Delta_0^\perp),
\]
the representation $\omega_0$ pulls back to $\omega_{0,Q}\boxtimes \omega_{0,Z}$, and, under the homomorphism
\[
\Mp(\Gamma\otimes_F Q')\times \Mp(\Gamma\otimes_F Z')\longrightarrow \Mp(\Gamma\otimes_F \Delta_1^\perp),
\]
the representation $\omega_1$ pulls back to $\omega_{1,Q'}\boxtimes \omega_{1,Z'}$.

Recall that, in Section~\ref{group g2 parabolics}, we defined $Q_1$ as the three-step parabolic of $G$. It has a Levi decomposition $Q_1=L_1U_1$, with Levi subgroup $L_1\simeq \GL_2$ and unipotent subgroup $U_1$ equipped with a three-step filtration $U_1=U_1(1)\supset U_1(2)\supset U_1(3)\supset U_1(4)= 1$ described in \cite[\S3.1]{BS}. In particular, the quotient $U_1/U_1(3)$ is isomorphic to the Heisenberg group associated with a two-dimensional symplectic space. This yields a natural action of $U_1/U_1(3)$ on $\omega_{0,Z}$ and $\omega_{1,Z'}$ via the Heisenberg representation. From now on, we consider the action of $U_1/U_1(3)$ on $\omega_0=\omega_{0,Q}\boxtimes \omega_{0,Z}$ and $\omega_1=\omega_{1,Q'}\boxtimes \omega_{1,Z'}$ given by the trivial representation on the first factors and the Heisenberg representation on the second factors.

There is also a natural action of the metaplectic group $\Mp_2$ on each of the representations $\omega_{0,Q}$, $\omega_{0,Z}$, $\omega_{1,Q'}$ and $\omega_{1,Z'}$. Moreover, the resulting action of $\Mp_2$ on $\omega_0=\omega_{0,Q}\boxtimes \omega_{0,Z}$ and $\omega_1=\omega_{1,Q'}\boxtimes \omega_{1,Z'}$ factors through $\SL_2$. From now on, we let $\SL_2\subset L_1$ act on $\omega_0$ and $\omega_1$ via this action.

Finally, there is a natural action of $\SO(Q)$ (resp. $\SO(Q')$) on $\omega_{0,Q}$ (resp. $\omega_{1,Q'}$), and we let $\U(V_2)$ (resp. $\U(V_2')$) act on $\omega_0=\omega_{0,Q}\boxtimes \omega_{0,Z}$ (resp. $\omega_1=\omega_{1,Q'}\boxtimes \omega_{1,Z'}$) via the embedding $\U(V_2)/Z(\U(V_2))\xhookrightarrow{} \SO(Q)$ (resp. $\U(V_2')/Z(\U(V_2'))\xhookrightarrow{} \SO(Q')$).

It follows from \cite[\S3.1.1]{BS} that, as a representation of $(\SL_2\ltimes U_1/U_1(3))\times\PU_3$,
\[
\Pi_{U_1(2),\psi'}=\mathrm{c}\mbox{-}{\Ind}_{\U(V_2)}^{\PU_3} \omega_0\oplus \mathrm{c}\mbox{-}{\Ind}_{\U(V_2')}^{\PU_3} \omega_1.
\]

Now we use this description of $\Pi_{U_1(2),\psi'}$ to prove the following result.

\begin{theorem}\label{thm: nonvanishing}
Let $\tau$ be a smooth irreducible representation of $\PU_3$ and let $\tau^\vee$ be its contragredient representation. Assume that either of the following two conditions holds:
\begin{itemize}
    \item The restriction of $\tau^\vee$ to $\U(V_2)$ has an irreducible quotient with trivial central character.
    \item The restriction of $\tau^\vee$ to $\U(V_2')$ has an irreducible quotient with trivial central character.
\end{itemize}
Then, the theta lift $\Theta_{\PU_3}(\tau)$ is nonzero.
\end{theorem}
\begin{proof}
To prove the non-vanishing of the theta lift, it suffices to prove that the space
\[
\Hom_{\PU_3}(\Pi_{U_1(2),\psi'},\tau)=\Hom_{\PU_3}(\mathrm{c}\mbox{-}{\Ind}_{\U(V_2)}^{\PU_3} \omega_0,\tau)\oplus \Hom_{\PU_3}(\mathrm{c}\mbox{-}{\Ind}_{\U(V_2')}^{\PU_3} \omega_0,\tau)
\]
is nontrivial.

Assume that the restriction of $\tau^\vee$ to $\U(V_2)$ has an irreducible quotient $\rho$ with trivial central character. Let $\rho^\vee$ be the contragredient of $\rho$ as a smooth representation of $\U(V_2)$. Note that $\rho^\vee$ can be regarded as a representation of $\PU(V_2)\simeq\SO(Q)^+$. Recall that $\omega_0=\omega_{0,Q}\boxtimes \omega_{0,Z}$, and $\U(V_2)$ acts on this space via the action of $\SO(Q)^+$ on the factor $\omega_{0,Q}$. Since any irreducible representation of $\SO(Q)$ (and \emph{a fortiori} of $\SO(Q)^+$) appears as a quotient of the Weil representation $\omega_{0,Q}$, so does the representation $\rho^\vee$. Hence, we can regard $\rho^\vee$ as a $\U(V_2)$-quotient of $\omega_0$, and, thus, we can also regard $\mathrm{c}\mbox{-}{\Ind}_{\U(V_2)}^{\PU_3} \rho^\vee$ as a $\PU_3$-quotient of $\mathrm{c}\mbox{-}{\Ind}_{\U(V_2)}^{\PU_3} \omega_0$. Therefore, to conclude the proof in this case we just need to show that $\tau$ is a $\PU_3$-quotient of $\mathrm{c}\mbox{-}{\Ind}_{\U(V_2)}^{\PU_3} \rho^\vee$, i.e.
\[
\Hom_{\PU_3}(\mathrm{c}\mbox{-}{\Ind}_{\U(V_2)}^{\PU_3} \rho^\vee,\tau)\neq 0.
\]
But, dualizing the left-hand side of the previous expression and applying Frobenius reciprocity, it becomes
\[
\Hom_{\U(V_2)}(\tau^\vee,\rho),
\]
which is nonzero since $\rho$ is a $\U(V_2)$-quotient of $\tau^\vee$.

The case in which the restriction of $\tau^\vee$ to $\U(V_2')$ has an irreducible quotient with trivial central character can be dealt with in a similar way.

\end{proof}

\subsection{A see-saw argument}\label{subsec: see saw argument}

In this subsection, we use a see-saw argument to prove that the contragredient of the representation $\sigma^+$ of $\U_3$ defined above, when restricted to a representation of a suitable 2-variable unitary group, has a quotient with trivial central character. We also prove this for the contragredient of $\sigma^-$ under the assumption that $\chi^2\neq 1$. Combined with the results in the previous subsection, this will conclude the proof of the non-vanishing of the corresponding theta lifts to the group $G$.

Recall that $\chi$ is a conjugate-symplectic character of $K^\times$. In general, if $W$ is an $m$-dimensional skew-Hermitian space over $K$ and $V$ is an $n$-dimensional Hermitian space over $K$, the choice of characters $\chi_W=\chi^m$ and $\chi_V=\chi^n$ determines a lift $\tilde{\iota}_{\chi,\psi}:\U(V)\times \U(W)\rightarrow \mathrm{Mp}(V\otimes_K W)$ of the natural homomorphism $\iota:\U(V)\times \U(W)\rightarrow \mathrm{Sp}(V\otimes_K W)$, and we can use this lift to define $\Omega_{V,W}:=\tilde{\iota}_{\chi,\psi}^\ast \omega_\psi$. Throughout this subsection, these are the representations that we use to define theta lifts between different unitary groups.

Let $V_3$ be a 3-dimensional Hermitian space over $K$. Assume that its Hermitian form, with respect to some $K$-basis, is defined by the matrix
\[
\Phi_3 = 
\begin{pmatrix}
0 & 0 & -1 \\
0 & -1 & 0 \\
-1 & 0 & 0
\end{pmatrix}.
\]
As before, let $V_2$ and $V_2'$ be 2-dimensional Hermitian spaces such that $V_2$ is split and $V_2'$ is nonsplit. Then, we can write
\begin{align*}
    V_3\coloneqq V_1\oplus V_2 \cong V_1'\oplus V_2',
\end{align*}
where $V_1$ and $V_1'$ are non-isomorphic 1-dimensional Hermitian spaces over $K$. Let $W$ and $W'$ be 1-dimensional skew-Hermitian spaces over $K$ such that $\Theta_{V_1,W}(\mathbb{1}_W)\neq 0$ and $\Theta_{V_1',W'}(\mathbb{1}_{W'})\neq 0$. Observe that, by dichotomy, the spaces $W$ and $W'$ are not isomorphic. More precisely, we have that
\[
\epsilon(V_1)\epsilon(W)=\epsilon_K(1/2,\chi,\psi(\mathrm{Tr}_{K/F}(-\delta(\cdot)))=\epsilon(V_1')\epsilon(W'),
\]
where $\delta\in K^\times$ denotes a trace-zero element which we also use to define the signs $\epsilon(W)$ and $\epsilon(W')$ as in \eqref{epsilon w}.

We have that $\Theta_{V_3,W}(\mathbb{1}_W)$ and $\Theta_{V_3,W'}(\mathbb{1}_{W'})$ are both nonzero, because we are in the stable range. Furthermore, it follows from Proposition~\ref{thetaliftinWitttower} and Theorem~\ref{HarrisKudlaSweet} that
\begin{align*}
\Theta_{V_3,W}(\mathbb{1}_W) \text{ is non-tempered } & \Longleftrightarrow \epsilon_K(1/2,\chi,\psi(\mathrm{Tr}_{K/F}(-\delta(\cdot))) \epsilon_K(1/2,\chi^3,\psi(\mathrm{Tr}_{K/F}(-\delta(\cdot)))=1  \\ & \Longleftrightarrow \Theta_{V_3,W'}(\mathbb{1}_{W'}) \text{ is supercuspidal.}
\end{align*}
In particular, when $\chi^2=1$, the statements above always hold.

We define $\sigma=\Theta_{V_3,W}(\mathbb{1}_W)$ and $\sigma'=\Theta_{V_3,W'}(\mathbb{1}_{W'})$. Then, we have that $\lbrace \sigma,\sigma'\rbrace =\lbrace \sigma^+,\sigma^-\rbrace$, and, if $\chi^2=1$, then $\sigma=\sigma^+$ and $\sigma'=\sigma^-$.

\begin{proposition}\label{prop: 6.3}
With the previous definitions, it always holds that $\Theta_{V_2,W}(\mathbb{1})\neq 0$, and
 \begin{align*}
  \Theta_{V_2',W}(\mathbb{1}_W)\neq 0 \text{ if and only if } \chi^2\neq 1.
 \end{align*}
These statements also hold if we replace $W$ by $W'$.
\end{proposition}
\begin{proof}
Observe that $V_2'$ is the first element in its Witt tower, whereas $V_2$ is the second element in its Witt tower, the first one being a zero-dimensional Hermitian space. The trivial character on $\U(W)$ (or on $\U(W')$) is said to lift in dimension zero if and only if $\chi^2=1$. Therefore the result follows from the conservation relation in Theorem \ref{thm: conservation law}.
\end{proof}

We denote by $\sigma^\vee$ and $(\sigma')^\vee$ the contragredient of the smooth $\PU_3$-representations $\sigma$ and $\sigma'$, respectively.

\begin{proposition}\label{prop:subrep with trivial central char}
    The restriction of the representation $\sigma^\vee$ to $\U(V_2)$ has a unique irreducible quotient with trivial central character. 
\end{proposition}
\begin{proof}
 Consider the see-saw diagram
 \begin{center}
    \begin{tikzcd}
        \U(V_3) \arrow[dash,rd]   & \U(W)\times \U(W) \arrow[dash,ld] \\
        \U(V_1)\times \U(V_2) & \Delta \U(W)
    \end{tikzcd}
\end{center}
and recall that we choose $\chi_{V_i}=\chi^i$ and $\chi_{W}=\chi$ as splitting characters. This choice is compatible with respect to the see-saw diagram above. Let $G_1=\mathrm{U}(V_1)\times \mathrm{U}(V_2)$ and $H_1=\Delta \mathrm{U}(W)$. Let $\tau$ be an irreducible representation of $\U(V_2)$ with trivial central character. Then, we have
\begin{align*}
    \Hom_{G_1}\left(\sigma, \mathbb{1}_{\U(V_1)}\boxtimes\tau\right)
    &\cong \Hom_{G_1\times H_1}\left(\Omega, \mathbb{1}_{\U(V_1)}\boxtimes \tau \boxtimes \mathbb{1}_{\U(W)}\right) \\
    &\cong \Hom_{H_{1}}\left(\mathbb{1}_{\U(W)}\otimes \Theta_{V_2,W}(\tau),\mathbb{1}_{\U(W)}\right).
\end{align*}
Since $\tau$ has trivial central character, the big theta lift $\Theta_{V_2,W}(\tau)$ is zero unless $\tau=\theta_{V_2,W}(\mathbb{1}_{\U(W)})$, which is nonzero because of Proposition~\ref{prop: 6.3}. In this case, $\Theta_{V_2,W}(\tau)=\mathbb{1}_{\U(W)}$ and therefore the space above is one-dimensional.

It follows from the previous discussion that the restriction of $\sigma$ to $\U(V_2)$ has a unique irreducible quotient with trivial central character, which is isomorphic to $\theta_{V_2,W}(\mathbb{1}_{\U(W)})$. By \cite[Lemma~2.1(ii)]{HKS}, we know that $\sigma^\vee\simeq \sigma^c$, with the notation of Remark~\ref{remark: extensions of representations}. Therefore, we conclude that the restriction of $\sigma^\vee$ to $\U(V_2)$ has a unique irreducible quotient with trivial central character. 
\end{proof}

\begin{proposition}
     Assume that $\chi^2\neq 1$. Then, the restriction of the representation $(\sigma')^\vee$ to $\U(V_2')$ has a unique irreducible quotient with trivial central character. 
\end{proposition}
\begin{proof}
    This can be proved by a see-saw argument as in the previous proposition, replacing $V_1$, $V_2$, and $W$, with $V_1'$, $V_2'$, and $W'$, respectively. Note that the hypothesis $\chi^2\neq 1$ is needed to ensure that the theta lift $\tau=\theta_{V_2',W'}(\mathbb{1}_{\U(W')})$ is nonzero, which follows in this case from Proposition~\ref{prop: 6.3}. When running the previous argument, with $H_1'=\Delta\U(W')$, this yields
    \[
    \Hom_{H_{1}'}\left(\mathbb{1}_{\U(W')}\otimes \Theta_{V_2',W'}(\tau),\mathbb{1}_{\U(W')}\right)=\Hom_{H_{1}'}\left(\mathbb{1}_{\U(W')}\otimes \mathbb{1}_{\U(W')},\mathbb{1}_{\U(W')}\right),
    \]
    which is one-dimensional.
\end{proof}

Now, an application of Theorem~\ref{thm: nonvanishing} yields the following result.
\begin{theorem}\label{thm: nonvanishing of pi}
With the previous definitions, we have the following:
\begin{enumerate}
    \item The representation $\Theta_{\PU_3}(\sigma^+)$ of $G$ is nonzero, and thus $\pi^+$ is nonzero. 
    \item If $\chi^2\neq 1$, the representation $\Theta_{\PU_3}(\sigma^-)$ of $G$ is nonzero, and thus $\pi^-$ is nonzero. 
\end{enumerate}
\end{theorem}

\section{Vanishing of theta lifts}\label{section: vanishing of theta lifts}

Let $K/F$ be a quadratic extension of local fields of characteristic zero. Fix a nontrivial additive character $\psi:F\rightarrow \bb{C}^\times$ and a conjugate-symplectic character $\chi:K^\times\rightarrow \bb{C}^\times$. Let $\sigma^+$ and $\sigma^-$ be the representations of $\PU_3$ introduced in Definition~\ref{big theta lift howe ps} with respect to these data and let $\pi^+$ and $\pi^-$ be their corresponding small theta lifts to the group $G$, as introduced in Definition~\ref{definition theta lift g2}. In this section we will prove some vanishing results for the representation $\pi^-$. 


In Section~\ref{subsec: vanishing of the generic part}, we prove that $\pi^-$ does not have any generic subquotient. This result, combined with the non-vanishing of $\pi^-$ when $\chi^2 \neq 1$ proved in Theorem~\ref{thm: nonvanishing of pi}, implies that $\pi^-$ is irreducible if $\chi^2 \neq 1$. This completes the proof of the second part of Theorem~\ref{theorem: pi+ for g2}. The rest of the section considers the case when $\chi^2 = 1$ and proves the vanishing of $\pi^-$ in that situation, completing the proof of the third point of Theorem~\ref{theorem: pi+ for g2}. For that, we study the twisted coinvariant spaces of $\pi^\pm$ corresponding to generic characters of the unipotent radical of the Heisenberg parabolic subgroup of $G$. More precisely, for every such character $\psi_E$:  
\begin{itemize}
    \item In Section~\ref{subsec: Twisted coinvariant spaces for pi+}, we obtain an explicit description of $\pi^+$ which allows us to verify that the twisted coinvariant space for $\pi^+$ with respect to $\psi_E$ does not vanish.
    \item In Section \ref{subsec: vanishing of the non-generic part}, we express the twisted coinvariant space for $\pi^+ \oplus \pi^-$ with respect to $\psi_E$ as a sum of toric periods for $\sigma^+ \oplus \sigma^-$. Since the representations $\sigma^{\pm}$ are theta lifts of characters in $\U_1$, the non-vanishing of these periods can be expressed in terms of local epsilon factors, as it is done in \cite{BFGYYZ}. Moreover, in [\emph{loc.\,cit.}], it is proven that exactly one of these toric periods contributes to the sum with precisely dimension 1.     
\end{itemize}
As a consequence, we obtain that the twisted coinvariant spaces for $\pi^-$ with respect to any generic character of the unipotent radical of the Heisenberg parabolic subgroup of $G$ are zero, from which we deduce that $\pi^-$ does not have a nontrivial non-generic quotient. Combining this fact with the results in Section~\ref{subsec: vanishing of the generic part}, we conclude the vanishing of $\pi^-$. 


\subsection{Vanishing of the generic part}\label{subsec: vanishing of the generic part}

In this subsection, we prove that $\pi^-$ does not have any generic subquotient. When $\chi^2\neq 1$, we deduce from this that $\pi^-$ is irreducible. 
The proofs rely on results of \cite[\S4.3-4.4]{BS}. 

As before, let ${B}'$ be a Borel subgroup of $\PU_3$ and let $U'$ be its unipotent radical. It has a filtration $0\subseteq U'(2)\subseteq U'$, where $U'(2)\simeq F$ and $U'/U'(2)\simeq K$. Via the last identification, we can define a character $\psi_{U'}$ of $U'$ corresponding to a nontrivial additive character $\psi\circ \Tr_{K/F}$ of $K$.

According to \cite[Corollary~4.6]{BS}, in order to prove the vanishing of the generic part of $\pi^-$, it suffices to prove the vanishing of the space of Whittaker periods
\[
\Hom_{U'}((\sigma^{-})^\vee,\bar{\psi}_{U'})\subseteq \Hom_{U'(2)}((\sigma^{-})^\vee,\bb{1}).
\]
Note that, as explained in the proofs of Proposition \ref{prop: irreducibility of induction to G'} and of Proposition \ref{prop: reducibility of induction to G'}, the representation $(\sigma^-)^\vee$ is the big theta lift of a character of $\U_1$ to a representation of $\U_3$. In particular, we can use the explicit description of the coinvariant spaces of theta lifts in this setting described in \cite[Section 2.12, Guided Exercise (ii)]{AWS} to affirm that
\[
(\sigma^-)^\vee_{U'(2)}=(\sigma^-)^\vee_{U'}.
\]
Since $(\sigma^-)^\vee$ is supercuspidal, as $(\sigma^-)^\vee \simeq (\sigma^-)^c$ by \cite[Lemma 2.1(ii)]{HKS} and $\sigma^-$ is supercuspidal, the right-hand side is equal to zero, implying the vanishing of the space of Whittaker periods.

We write the result as a proposition for reference.

\begin{propo}\label{prop: generic vanishing}
    The representation $\pi^-$ does not have any generic subquotient.
\end{propo}

From there, we obtain the following result.

\begin{theorem}\label{thm: pi- is irreducible if chi^2 neq 1}
    If $\chi^2\neq 1$, the representation $\pi^-$ is irreducible and tempered.
\end{theorem}
\begin{proof}
    Note that $\Ind_{\PU_3}^{G'} \sigma^-$ is irreducible (by Proposition~\ref{prop: irreducibility of induction to G'}) and tempered (because $\sigma^-$ is supercuspidal).  We also make the following two observations: 
    \begin{itemize}
        \item By \cite[Remark~4.4]{BS}, every irreducible quotient of $\pi^-$ is tempered.
        \item By Proposition~\ref{prop: generic vanishing}, every irreducible quotient of $\pi^-$ is non-generic.
    \end{itemize}    
    Now, Theorem~\ref{thm: nonvanishing of pi} implies that there exists a nonzero irreducible quotient of $\pi^-$. Moreover, \cite[Proposition~4.15]{BS} together with these two observations imply that such quotient is unique. Hence, $\pi^-$ is irreducible and tempered as desired.
\end{proof}

\subsection{Twisted coinvariant spaces for $\pi^+$}\label{subsec: Twisted coinvariant spaces for pi+}
We assume from now on that $\chi^2 = 1$. This implies that $\chi$ is $\Gal(K/F)$-invariant, hence there exists a character $\mu$ of $F^\times$ such that $\chi = \mu \circ \Norm_{K/F}$. Fix a choice of such character $\mu$. Note that $\mu \neq \mu^{-1}$ as $\mu^2 = \omega_{K/F}$. Denote by $P$ the Borel of $\GL_2(F)$ consisting on upper triangular matrices.

\begin{proposition}\label{prop: automorphic induction of chi when chic = chi}
	Consider the same notation as above. We have: 
	\begin{enumerate}
		\item The automorphic induction of $\chi\lvert \Norm_{K/F}( \ ) \rvert^{1/2}$ from $K^\times$ to $\GL_2(F)$ is the irreducible representation $\vert \det \vert^{1/2} \tau$, where 
        \[ 
        \tau = i_P^{\GL_2(F)}(\mu, \mu^{-1}).
        \]
		\item Regard $\vert \det \vert^{1/2}\tau$ as a representation of $L_1 \simeq \GL_2$, i.e. 
		\begin{equation}\label{deftau}
		\vert \det \vert^{1/2}\tau = i_{T(U\cap L_1)}^{L_1}(\mu\lvert \ \rvert^{1/2} \otimes \mu^{-1} \lvert \ \rvert^{1/2}).
		\end{equation}
		Then, $\pi^+$ is the unique irreducible quotient of $i_{Q_1}^G(\vert \det \vert^{1/2}\tau)$.
	\end{enumerate}
\end{proposition}

\begin{proof}
    We start proving the first point. Note that $(\mu \circ \Norm_{K/F})\lvert \Norm_{K/F}(\ )\rvert^{1/2}$ is Galois invariant. Thus, the automorphic induction of this character to $\GL_2(F)$ is the irreducible principal series corresponding to the character of the diagonal torus in $\GL_2(F)$ given by $\mu| \ |^{1/2} \otimes \omega_{K/F}\mu |\ |^{1/2}$. Note that we used that, since $\omega_{K/F} \neq \vert \ \vert^{\pm 1}$, such principal series is irreducible. Using the relation $\mu^2 = \omega_{K/F}$, we can rewrite the automorphic induction as 
    \[
    i_P^{\GL_2(F)}(\mu\lvert \ \rvert^{1/2}, \mu^{-1} \lvert \ \rvert^{1/2}) = \vert \det \vert^{1/2}i_P^{\GL_2(F)}(\mu, \mu^{-1}),
    \]
    proving the first point. Since $\pi^+ = \theta_{\PU_3}(\sigma^+)$, and we saw in the proof of Theorem~\ref{theorem: pi+ for g2} that $i_{Q_1}^G(\vert \det \vert^{1/2}\tau)$ has a unique irreducible quotient, the second point of the proposition follows from the first one, Proposition~\ref{prop: exceptional theta lift non tempered representation} and Remark~\ref{rmk: tau is the automorphic induction of chi}.
\end{proof}

In view of the previous proposition, the next proposition will provide a description of $\pi^+$ in this case.

\begin{proposition}\label{From3stepparabolictoHeisenberg}
	The unique irreducible quotient of $i_{Q_1}^G(\vert \det \vert^{1/2}\tau)$ is given by $i_{Q_2}^G(\mu \circ \det)$.
\end{proposition}
\begin{proof}
    The key to prove this proposition is to rewrite normalized parabolic inductions for the parabolic $Q_1$ (resp. $Q_2$) in terms of the parabolic $Q_2$ (resp. $Q_1$) and to use intertwining operators.
    
    Using the isomorphism induced by \eqref{isoL1GL2}, it can be seen
	\begin{equation}\label{eq: from parabolic to Borel}
	i_{Q_1}^G(\vert \det \vert^{1/2}\tau) = i_{Q_1}^G i_{T(U\cap L_1)}^{L_1} (\mu\lvert \ \rvert^{1/2}  \otimes \mu^{-1} \lvert \ \rvert^{1/2}) = i_{Q_1}^G i_{B}^{Q_1} (\mu\lvert \ \rvert^{1/2}(\alpha +\beta) \cdot \mu^{-1} \lvert \ \rvert^{1/2}(\alpha)),
	\end{equation}
	where in the last expression we are extending $\mu\lvert \ \rvert^{1/2}(\alpha +\beta) \cdot \mu^{-1} \lvert \ \rvert^{1/2}(\alpha)$ from a character of $T$ to a character of $B$. Now, using that $i_{Q_1}^G i_{B}^{Q_1} = i_{B}^G$, we obtain	
	\[
	i_{Q_1}^G(\vert \det \vert^{1/2}\tau) = i_B^{G}\left(\mu \lvert \ \rvert^{1/2}(\alpha + \beta) \cdot \mu^{-1}\lvert \ \rvert^{1/2}(\alpha)\right).
	\]
	This can be rewritten as
	\begin{equation}\label{eq: change of variables alpha and beta}
    \begin{split}
	i_B^{G}\left(\mu \lvert \ \rvert^{1/2}(\alpha + \beta) \cdot \mu^{-1}\lvert \ \rvert^{1/2}(\alpha)\right) & =  i_B^G\left(\mu^{-1}\lvert \ \rvert^{1/2}(2\alpha  + \beta) \cdot \mu^2(\alpha + \beta) \right) \\ & = i_{Q_2}^G i_{T(U\cap L_2)}^{L_2}\left(\mu^{-1}\lvert \ \rvert^{1/2} \otimes \mu^2 \right),
 \end{split}
 \end{equation}
	where in the last equality we proceeded in a similar way as in \eqref{eq: from parabolic to Borel} and used the isomorphism induced by \eqref{isloL2GL2}. By Proposition 1.1 (i) of \cite{M}, the representation of $L_2 \simeq \GL_2$ given by $i_{T(U\cap L_2)}^{L_2}\left(\mu^{-1}\lvert \ \rvert^{1/2} \otimes \mu^2 \right)$ is irreducible and we have an isomorphism
	\[
	i_{T(U\cap L_2)}^{L_2}\left(\mu^{-1}\lvert \ \rvert^{1/2}, \mu^2 \right) \simeq i_{T(U\cap L_2)}^{L_2}\left(\mu^2, \mu^{-1}\lvert \ \rvert^{1/2} \right).
	\]
	Hence, 
\begin{equation*}
\begin{split}        
	i_{Q_2}^G i_{T(U\cap L_2)}^{L_2}\left(\mu^{-1}\lvert \ \rvert^{1/2} \otimes \mu^2 \right) & \simeq i_{Q_2}^G i_{T(U\cap L_2)}^{L_2}\left(\mu^2 \otimes \mu^{-1}\lvert \ \rvert^{1/2} \right) \\ & \simeq i_B^G\left( \mu^2(2\alpha + \beta) \cdot \mu^{-1}\lvert \ \rvert^{1/2} (\alpha + \beta)  \right)  \simeq i_B^G \left( \mu \lvert \ \rvert^{1/2}(\alpha + \beta) \cdot \mu^2(\alpha)  \right),
 \end{split}
\end{equation*}
	where we used that $2\alpha + \beta = \alpha + (\alpha + \beta)$. Now, we can use Proposition 1.1 (i) of \cite{M} to obtain
	\[
	i_B^G \left( \mu \lvert \ \rvert^{1/2}(\alpha + \beta) \cdot \mu^2(\alpha)  \right) \simeq i_{Q_1}^{G}i_{T(U\cap L_1)}^{L_1}\left( \mu \lvert \ \rvert^{1/2} \otimes \mu^2  \right) \simeq i_{Q_1}^{G}i_{T(U\cap L_1)}^{L_1}\left(\mu^2 \otimes \mu \lvert \ \rvert^{1/2} \right).
	\]
    Proceeding in a similar way as above, we deduce
	\begin{equation*}
 \begin{split}
	i_{Q_1}^{G}i_{T(U\cap L_1)}^{L_1} & \left(\mu^2 \otimes \mu \lvert \ \rvert^{1/2} \right)  \simeq i_B^G\left( \mu^2(\alpha + \beta) \cdot \mu \lvert \ \rvert^{1/2}(\alpha) \right) \\ & \simeq i_B^G\left(  \mu \lvert \ \rvert^{1/2}(2\alpha + \beta) \cdot \mu \lvert \ \rvert^{-1/2}(\alpha + \beta) \right) \simeq i_{Q_2}^G i_{T(U\cap L_2)}^{L_2}\left(  \mu \lvert \ \rvert^{1/2} \otimes \mu \lvert \ \rvert^{-1/2} \right).
 \end{split}
    \end{equation*}
	By Proposition 1.1 (ii) of \cite{M}, we have that $i_{T(U \cap L_2)}^{L_2} \left( \mu\lvert \ \rvert^{1/2} \otimes \mu \lvert  \ \rvert^{-1/2} \right)$ has a unique irreducible quotient given by $\mu \circ \det$. It follows that $i_{Q_1}^G (\vert \det \vert^{1/2}\tau) = i_{Q_2}^G i_{T(U \cap L_2)}^{L_2} \left( \mu\lvert \ \rvert^{1/2} \otimes \mu \lvert  \ \rvert^{-1/2} \right)$ has a quotient given by $i_{Q_2}^G \left(\mu \circ \det\right)$. Finally, note that $\mu^2 = \omega_{K/F} \neq 1$ so, in particular, $\mu$ is unitary. Hence, Theorem 3.1 (i) of \cite{M} affirms that $i_{Q_2}^G \left(\mu \circ \det\right)$ is irreducible, and we are done. 
\end{proof}
We therefore obtain the following corollary.
\begin{corollary}\label{prop: quotient of pi_v+}
    Suppose that $\chi^2=1$ and let $\mu$ be a character of $F^\times$ such that $\chi=\mu\circ\norm_{K/F}$. Then, 
    \[
        \pi^+ = i_{Q_2}^{G}(\mu \circ \det),
    \]
    where $Q_2$ denotes the Heisenberg parabolic of $G_2$.
\end{corollary}
\begin{proof}
    This follows from Proposition~\ref{prop: automorphic induction of chi when chic = chi} and Proposition~\ref{From3stepparabolictoHeisenberg}.
\end{proof}
\begin{remark}\label{remark: pi+ is irreducible}
	Even though we will not need it explicitly, it can be verified that in fact $\Theta_{\PU_3}(\sigma^+) = i_{Q_2}^G(\mu \circ \det)$.
\end{remark}

We use the previous corollary to determine the non-vanishing of the twisted coinvariant spaces of $\pi^+$.

\begin{proposition}\label{prop: coinvariants of pi plus}
	Consider the same notation as above. Let $\psi'$ be a character of $U_2$. We have
	\[
	\dim\left(\Hom_{U_2}(\pi^+, \psi')\right) \geq 1. 
	\]
\end{proposition}
\begin{proof}
	By Corollary~\ref{prop: quotient of pi_v+}, it is enough to see that 
	\[
	\dim\left(\Hom_{U_2}( i_{Q_2}^G(\mu \circ \det), \psi')\right) \geq 1.
	\]
	Using the definition of normalized induction, we have 
	\[
	i_{Q_2}^G(\mu \circ \det) = \left\{f: G_2 \to \bb{C} \mid f(qx) = (\mu \circ \det)(q) \delta_{Q_2}^{1/2}(q) f(x) \text{ for } q \in Q_2, \ x \in G  \right\},
	\]
	where we consider locally constant functions. The group $G$ acts on the left of this space as follows: if $g \in G$ and $f \in	i_{Q_2}^G(\mu \circ \det)$, then $(g\cdot f)(x) = f(xg)$ for every $x \in G$. Restrict the representation $i_{Q_2}^G(\mu \circ \det)$ to a representation of $U_2$. Let $C^\infty_c(U_2)$ be the regular representation (consisting on smooth functions with compact support) of $U_2$. Then, we have an injective map of $U_2$-representations
	\[
	C^\infty_c(U_2) \xhookrightarrow{} i_{Q_2}^G(\mu \circ \det).
	\]
    We proceed to describe it. The Bruhat decomposition of $G_2$ with respect to the parabolic subgroup $Q_2$ is 
	\[
	G_2 = Q_2 \cup Q_2 w_\beta Q_2 \cup Q_2 w_{\beta \alpha \beta} Q_2 \cup Q_2 w_{\beta\alpha\beta\alpha\beta} Q_2
	\]
	(see \cite[Page 289]{JR}). Hence, we can find the following $U_2$-invariant subspace of $i_{Q_2}^G(\mu \circ \det)$
    \begin{multline*}
    \left\{ f: Q_2 w_{\beta\alpha\beta\alpha\beta} U_2 \to \bb{C} \mid f(qx) = (\mu \circ \det)(q)\delta_{Q_2}^{1/2}(q)f(x) \text{ for } q \in Q_2, \ x \in  Q_2 w_{\beta\alpha\beta\alpha\beta} U_2 \right\} \\
    \simeq \left\{f: w_{\beta\alpha\beta\alpha\beta} U_2 \to \bb{C} \right\},
    \end{multline*}
    where we again consider locally constant functions. Note that we used that the big cell in the Bruhat decomposition of $G_2$ given above is the one corresponding to $w_{\beta\alpha\beta\alpha\beta}$ and therefore we have a natural bijection $Q_2 w_{\beta\alpha\beta\alpha\beta} U_2\simeq Q_2\times U_2$. If we further restrict the second space to functions of compact support, we obtain the desired $U_2$-representation subspace  of $i_{Q_2}^G(\mu \circ \det)$ isomorphic to the regular representation $C_c^\infty(U_2)$. 
    
    Taking $(U_2, \psi')$-coinvariants, which preserve injectivity by smoothness, we get
    \[
    C_c^\infty(U_2)_{U_2, \psi'} 
    \xhookrightarrow{}
    \left(i_{Q_2}^G(\mu \circ \det)\right)_{U_2,\psi'}.
    \]
    Since, $C_c^\infty(U_2)$ is the regular representation, we have that  $C_c^\infty(U_2)_{U_2, \psi'}$ is $1$-dimensional, and the result follows from there. 	
\end{proof}

\begin{remark}
    Consider the same notation as above and suppose in addition that $\psi'$ is generic. It follows from \cite[Theorem 3]{JR} that $\dim\left(\Hom_{U_2}( i_{Q_2}^G(\mu \circ \det), \psi')\right) \leq 1$. Hence, we have $\dim\left(\Hom_{U_2}( i_{Q_2}^G(\mu \circ \det), \psi')\right) = 1$. 
\end{remark}

\subsection{Vanishing of the non-generic part}\label{subsec: vanishing of the non-generic part}
In this subsection, we show the vanishing of the non-generic part of $\pi^{-}$ when $\chi^2=1$ by proving that the corresponding generic Fourier coefficients along the unipotent $U_2$ are all zero.

Recall that  $J=J_3(K)$, the set of $3\times 3$ Hermitian matrices defined in Section~\ref{subsection: dual pair}. Let $P=MN$ be the Heisenberg maximal parabolic subgroup of $H$, where $H$ is as defined in Section~\ref{subsection: dual pair}, and let $\bar{P}=M\bar{N}$ be the parabolic subgroup opposite to $P$. Let $Z$ (resp. $\bar{Z}$) be the center of $N$ (resp. $\bar{N}$). Then $N/Z$ and $\bar{N}/\bar{Z}$ are commutative, and we have that $N/Z\cong \mathfrak{n}/\mathfrak{z}$ and $\bar{N}/\bar{Z}\cong \bar{\mathfrak{n}}/\bar{\mathfrak{z}}$, where $\mathfrak{n}$, $\mathfrak{z}$, $\bar{\mathfrak{n}}$ and $\bar{\mathfrak{z}}$ denote the Lie algebras of $N$, $Z$, $\bar{N}$ and $\bar{Z}$, respectively. According to \cite[(2.5)]{GS}, we have the following decompositions:
\begin{gather*}
   \mathfrak{n}/\mathfrak{z}=F\oplus J\oplus J^*\oplus F^* \\
   \bar{\mathfrak{n}}/\bar{\mathfrak{z}}=F^*\oplus J^*\oplus J\oplus F.
\end{gather*}
The Killing form on $\mathfrak{h}$ induces a non-degenerate pairing between $N/Z\cong \mathfrak{n}/\mathfrak{z}$ and $\bar{N}/\bar{Z}\cong \bar{\mathfrak{n}}/\bar{\mathfrak{z}}$ which is given by
\begin{align*}
    \langle\left(x, u, u^{*}, x^{*}\right),\left(y^{*}, v^{*}, v, y\right)\rangle=x y^{*}+\langle u, v^{*}\rangle+\langle v, u^{*}\rangle+y x^{*}.
\end{align*}
The following theorem gives a useful description of the minimal representation $\Pi$ of $H$.

\begin{theorem}\cite[Theorem 7.1]{MS}\label{thm: model of min rep}
Let $\Pi$ be the minimal representation of $H$. Let $\Pi_Z$ and $\Pi_N$ be the maximal $Z$-invariant and $N$-invariant quotients of $\Pi$. Then $\Pi_Z$ has a $P$-equivariant filtration
\begin{align*}
    0\to C_c^\infty(\bar{\Omega})\to \Pi_Z \to \Pi_N\to 0,
\end{align*}
where $\bar{\Omega}$ is the smallest nontrivial $M$-orbit in $\bar{N}/\bar{Z}$ and $C_c^\infty(\bar{\Omega})$ denotes the space of locally constant, compactly supported functions on $\bar{\Omega}$. 

Moreover, the action of $P$ on $C_c^\infty(\bar{\Omega})$ is given by
\begin{align*}
\pi(n) f(x)&=\psi(\langle n, x\rangle) f(x), \quad n \in N,\\
\pi(m) f(x)&=|\det(m)|^{1/5} f\left(m^{-1} x m\right), \quad m \in M,
\end{align*}
where $\det$ denotes the determinant of the representation of $M$ on $N/Z$.
\end{theorem}
 
Now we consider $G_2\cap P=L_2 U_2=\mathrm{GL}_2U_2$, the Heisenberg parabolic subgroup of $G_2$. Notice that 
$$
\bar{U}_{2}/\bar{Z} \cong F^\ast \oplus F e^{*} \oplus F e\oplus F \subset F^{*} \oplus J^{*} \oplus J \oplus F \cong \bar{N}/\bar{Z},
$$
where $e$ is an element of $J$ that is invariant under the action of $G'=\mathrm{PU}_3\rtimes \mu_2$, normalized by $\langle e^*,e\rangle=3$. An element $(a,b,c,d)\in \bar{U}_2/\bar{Z}$ defines a binary cubic form
\begin{align*}
    ax^3+bx^2y+cxy^2+dy^3,
\end{align*}
which defines a cubic separable algebra if and only if the corresponding $\GL_2$-orbit of $(a,b,c,d)$ is generic. Moreover, the restriction of $\langle\, ,\, \rangle$ defines a non-degenerate pairing between $U_2/Z$ and $\bar{U}_2/\bar{Z}$. Hence, a point $\bar{n}=(a,b,c,d)\in \bar{U}_2/\bar{Z}$ corresponding to a cubic separable algebra $E$ defines a character $\psi_E(x)\coloneqq \psi(\langle x, \bar{n}\rangle)$ on $U_2/Z$.

We define the twisted Jacquet module $\Pi_{U_2,\psi_E}$ to be the maximal quotient of $\Pi$ on which the action of $U_2$ is given by the character $\psi_E$. Theorem~\ref{thm: model of min rep} enables us to describe the $\PU_3$-module $\Pi_{U_2,\psi_E}$ explicitly as stated in the following proposition, which is essentially a summary of \cite[Proposition~2.8]{GS} and \cite[Lemma~2.9]{GS}.

\begin{proposition}\label{prop: X}
Let $E$ be a cubic separable $F$-algebra and let $X_E=\left\{A \in H_{3}^e(K)\;\colon\; f_{A}=f_{E}\right\}$, where 
\begin{itemize}
    \item $H_{3}^e(K)\coloneqq \{A\in M_3(K)\;\colon\; A=eA^\dagger e^{-1}\} $;
    \item $f_{A}$ is the characteristic polynomial of $A \in H_{3}^e(K)$; 
    \item $f_{E}$ is a fixed monic cubic polynomial such that $E \cong F[x] /\left(f_{E}(x)\right)$.
\end{itemize} 
Then, we have the following isomorphism of $\mathrm{PU}_3$-modules:
$$
\Pi_{U_2, \psi_{E}} \cong C_{c}^{\infty}(X_E).
$$
\end{proposition}
\begin{proof}
Since $\psi_E$ is not trivial, we have $(\Pi_N)_{U_2,\psi_E}=0$. Hence, taking $(U_2,\psi_E)$-coinvariants in the short exact sequence in Theorem~\ref{thm: model of min rep}, we obtain that $C_c^\infty(\bar{\Omega})_{U_2,\psi_E}\simeq \Pi_{U_2,\psi_E}$. 

Now we show that $C_c^\infty(\bar{\Omega})_{U_2,\psi_E}\simeq C_{c}^{\infty}(X_E)$. For $n\in U_2$ and $f\in  C_c^\infty(\bar{\Omega})$, we have 
\begin{align*}
\pi(n) f(x)&=\psi(\langle n, x\rangle) f(x) 
\end{align*}
by Theorem \ref{thm: model of min rep}. It follows from the proof of \cite[Lemma 2.9]{GS} that the natural projection map $C_c^\infty(\bar{\Omega})\twoheadrightarrow C_c^\infty(\bar{\Omega})_{U_2,\psi_E}$ can be realized as the restriction map $C_c^\infty(\bar{\Omega})\twoheadrightarrow C_c^\infty(X'_E)$, where
\[
X'_E=\lbrace A\in J \;\colon\; f_{Ae^{-1}}=f_E\rbrace\subseteq \bar{\Omega}.
\]
An element $z\in\PU_3$ acts on $X'_E$ by $A\mapsto zAz^\dagger$. The map $A\mapsto Ae^{-1}$ defines a bijection of $\PU_3$-sets between $X'_E$ and $X_E$, with $\PU_3$ acting on the latter by conjugation.
\end{proof}

In the statement and proof of the following lemma, we regard $H^e_3(K)$, $M_3(K)$ and any cubic separable $F$-algebra $E$ as Jordan algebras with multiplication defined by
\[
x\circ y=\frac{1}{2}(xy+yx).
\]
Of course, on $E$ this agrees with the standard multiplication.

\begin{lemma}\label{lemma: bijection}
    Let $E$ be a cubic separable $F$-algebra. The following $\PU_3$-sets are in natural bijection:
    \begin{enumerate}
        \item[(1)] the set $X_E$ defined above, with $\PU_3$ acting by conjugation;
        \item[(2)] the set of embeddings $i:E \xhookrightarrow{} H^{e}_{3}(K)$ of Jordan algebras over $F$, with $\PU_3$ acting by conjugation;
        \item[(3)] the set of embeddings $i:E\otimes_F K\xhookrightarrow{} M_3(K)$ of $K$-algebras which are compatible with the anti-involution defined by the nontrivial element $c\in\Gal(K/F)$ on each side, given on $E\otimes_F K$ by the natural Galois action and on $M_3(K)$ by $A\mapsto eA^{\dagger}e^{-1}$, with $\PU_3$ acting by conjugation.
    \end{enumerate}
\end{lemma}
\begin{proof}
    An embedding $i:E\xhookrightarrow{} H^e_3(K)$ of Jordan algebras over $F$ is determined by the image of the element $x\in E\cong F[x]/(f_E(x))$. Since $i$ is an embedding, the characteristic polynomial of $i(x)$ must be equal to $f_E$. This gives a natural bijection between (1) and (2).
    
    Observe that there is an isomorphism of $K$-algebras
    \[
    M_3(K)\otimes_F K\cong M_3(K)\times M_3(K)
    \]
    given on pure tensors by $A\otimes a\mapsto (aA, a\bar{A})$. Under this isomorphism, the action of the nontrivial element $c\in\Gal(K/F)$, which acts on $M_3(K)\otimes_F K$ via its natural action on the right factor $K$, is given on $M_3(K)\times M_3(K)$ by $(A,B)\mapsto (\bar{B},\bar{A})$. Moreover, the anti-involution on $M_3(K)\otimes_F K$ defined by $A\otimes a\mapsto eA^\dagger e^{-1}\otimes a$ corresponds to the anti-involution on $M_3(K)\times M_3(K)$ defined by $(A,B)\mapsto (e B^T e^{-1}, e A^T e^{-1})$. Therefore, the subspace $H^e_3(K)\otimes_F K$ maps to
    \[
    \lbrace (A, e A^T e^{-1})\;\colon\; A\in M_3(K)\rbrace \subseteq M_3(K)\times M_3(K),
    \]
    and, by projection onto the first factor, we obtain an isomorphism $H^e_3(K)\otimes_F K\cong M_3(K)$ of Jordan algebras over $K$. Moreover, under this isomorphism, the nontrivial element $c\in\Gal(K/F)$ acts on $M_3(K)$ by $A\mapsto e A^\dagger e^{-1}$, so that $M_3(K)^{\Gal(K/F)}$ is the set of matrices $H^e_3(K)\subseteq H^e_3(K)\otimes_F K\cong M_3(K)$. It follows from these remarks that there is a natural bijection between (2) and (3).
\end{proof}

We now prove that the sets appearing in the previous lemma are always non-empty.

\begin{lemma}
    There exists an embedding $i:E\otimes_F K\xhookrightarrow{} M_3(K)$ of $K$-algebras compatible with the anti-involution defined by the nontrivial element $c\in\Gal(K/F)$ on each side, given on $E\otimes_F K$ by the natural Galois action and on $M_3(K)$ by $A\mapsto eA^{\dagger}e^{-1}$.
\end{lemma}
\begin{proof}
    For any $\lambda\in F^\times$, we define the following Hermitian pairing on the $K$-vector space $E\otimes_F K$. For $x \otimes a, y \otimes b \in E\otimes_F K$, we define
    \[
    \langle x \otimes a,y \otimes b \rangle_\lambda = \lambda ab^c \Tr(xy),
    \]
    where $\Tr:E \to F$ denotes the trace form, and we extend the pairing additively. Choosing a suitable $\lambda$, the Hermitian space $(E\otimes_F K, \langle\,,\,\rangle_\lambda)$ is isomorphic to the Hermitian space $(V=K^3, \Phi_3=e)$. This provides an embedding
    \[
    E\otimes_F K \xhookrightarrow{} \mathrm{End}_K(E\otimes_F K) \simeq \mathrm{End}_K(V)=M_3(K).
    \]
    The actions of $c\in\Gal(K/F)$ on $E\otimes_F K$ and on $M_3(K)$ defined in the statement amount to taking the adjoint with respect to the corresponding Hermitian pairing. Therefore, the previous embedding is compatible with the action of $c$.
    
\end{proof}

Given a cubic separable $F$-algebra $E$, let $\underline{T}_{E,K}$ denote the algebraic group over $F$ such that, for an $F$-algebra $R$,
\[
\underline{T}_{E,K}(R)=\lbrace x\in E\otimes_F R \otimes_F K \;\colon\; \norm_{E\otimes_F R\otimes_F K/E\otimes_F R}(x)=1\rbrace,
\]
and let $T_{E,K}$ denote the corresponding group of $F$-points.
An embedding $i':E\otimes_F K\xhookrightarrow{} M_3(K)$ compatible with the anti-involution defined by the nontrivial element $c\in\Gal(K/F)$ on each side, as described in the previous lemma, defines an embedding of algebraic groups $i:\underline{T}_{E,K}\xhookrightarrow{} \underline{\U}_3$. For an $F$-algebra $R$, the corresponding map on $R$-points is the map
\[
(E\otimes_F K\otimes_F R)^1\xhookrightarrow{} \underline{\U}_3(R)
\]
obtained by restricting the map induced by $i'$ to norm-$1$ elements. Let $i'_0:E\otimes_F K\xhookrightarrow{} M_3(K)$ be an embedding compatible with the action of $c\in\Gal(K/F)$ and let $i_0$ be the corresponding embedding of algebraic groups $\underline{T}_{E,K}\xhookrightarrow{} \underline{\U}_3$. Let $\tilde{X}_E$ denote the set of embeddings of algebraic groups $\underline{T}_{E,K}\xhookrightarrow{} \underline{\U}_3$ which are conjugate to $i_0$ by an element of $\U_3(\bar{F})$.
The group $\PU_3$ acts on $\tilde{X}_E$ by conjugation.

\begin{propo}\label{prop: bijection}
    Let $E$ be a cubic separable $F$-algebra. There is a bijection of $\PU_3$-sets between $X_E$ and $\tilde{X}_E$.
\end{propo}
\begin{proof}
    By Lemma~\ref{lemma: bijection}, we just need to show that there is a bijection of $\PU_3$-sets between $\tilde{X}_E$ and the set of embeddings of $K$-algebras $i':E\otimes_F K\xhookrightarrow{} M_3(K)$ which are compatible with the anti-involution defined by the nontrivial element $c\in\Gal(K/F)$ on each side. Such a map can be extended to an embedding of $\bar{F}$-algebras
    \[
    E\otimes_F K\otimes_F \bar{F} \xhookrightarrow{} M_3(K\otimes_F \bar{F})
    \]
    compatible with the action of $c\in\Gal(K/F)$. Giving such a map is equivalent to giving an embedding of $\bar{F}$-algebras
    \[
    \bar{F}\times\bar{F}\times\bar{F} \xhookrightarrow{} M_3(\bar{F}).
    \]
    Therefore, it becomes clear that, after choosing an embedding $K\xhookrightarrow{} \bar{F}$, any two embeddings $i_1',i_2':E\otimes_F K\xhookrightarrow{} M_3(K)$ compatible with the action of $c\in\Gal(K/F)$ are conjugate by an element of $\U_3(\bar{F})\simeq\GL_3(\bar{F})$. In particular, any embedding $i':E\otimes_F K\xhookrightarrow{} M_3(K)$ compatible with the action of $c\in\Gal(K/F)$ is conjugate to $i_0'$ by an element of $\U_3(\bar{F})$. Moreover, if $a\in \U_3(\bar{F})$, then $a\cdot i_0' \cdot a^{-1}$ defines an embedding $E\otimes_F K\xhookrightarrow{} M_3(K)$ compatible with the action of $c\in\Gal(K/F)$ if and only if the cocycle $\tau\mapsto a^{-1}\cdot\tau(a)$ defines a cohomology class in
    \[
    \ker\left(H^1(F,\underline{T}_{E,K})\longrightarrow H^1(F,\underline{\U}_3)\right).
    \]
    This is the same condition that $a\in \U_3(\bar{F})$ must satisfy in order that $a\cdot i_0\cdot a^{-1}$ defines an embedding of algebraic groups $\underline{T}_{E,K}\xhookrightarrow{} \underline{\U}_3$. Furthermore, $a\in \U_3(\bar{F})$ fixes $i_0'$ if and only if $a\in i_0(\underline{T}_{E,K}(\bar{F}))$ 
    and thus if and only if it fixes $i_0$. Hence, it follows that the map defined above sending an embedding $i':E\otimes_F K\xhookrightarrow{} M_3(K)$ compatible with the action of $c\in\Gal(K/F)$ to an embedding of algebraic groups $i:\underline{T}_{E,K}\xhookrightarrow{} \underline{\U}_3$ provides the desired bijection.
\end{proof}

Let $\tilde{\cl{T}}_{E,K}$ denote set of $\PU_3$-orbits of $\tilde{X}_E$. 

\begin{corollary}
    Let $E$ be a cubic separable $F$-algebra. Then,
    \[
    \Pi_{U_2,\psi_E} \cong \bigoplus_{i \in \tilde{\mathcal{T}}_{E, K}} \mathrm{c}\mbox{-}{\Ind}_{i(T_{E, K})}^{\mathrm{PU}_{3}} \mathbb{1}.
    \]
\end{corollary}
\begin{proof}
    It follows from Proposition~\ref{prop: X} that
    \[
    \Pi_{U_2, \psi_{E}} \cong C_{c}^{\infty}(X_E).
    \]
    Proposition~\ref{prop: bijection} establishes a bijection of $\PU_3$-sets between $X_E$ and $\tilde{X}_E$. The stabilizer of an embedding of algebraic groups $i:\underline{T}_{E,K} \hookrightarrow \underline{\U}_3$ for the conjugation action of $\PU_3$ is $i(T_{E,K})$. Therefore, decomposing $X_E$ into $\PU_3$-orbits, we obtain
    \[
    C_{c}^{\infty}(X_E) \cong \bigoplus_{i \in \tilde{\mathcal{T}}_{E, K}} \mathrm{c}\mbox{-}{\Ind}_{i(T_{E, K})}^{\mathrm{PU}_{3}} \mathbb{1}.
    \]
\end{proof}

\begin{corollary}
    Let $E$ be a cubic separable $F$-algebra. Then,
    \[
    \dim \Hom_{U_2}(\Theta_{\PU_3}(\sigma^+),\psi_E)+\dim \Hom_{U_2}(\Theta_{\PU_3}(\sigma^-),\psi_E)=1.
    \]
\end{corollary}
\begin{proof}
    It follows from the previous corollary and an application of Frobenius reciprocity for compact induction that the left-hand side is equal to
    \[
   \sum_{{i \in \tilde{\mathcal{T}}_{E, K}}}\left( \dim\Hom_{i(T_{E,K})}((\sigma^+)^\vee,\bb{1})+\dim\Hom_{i(T_{E,K})}((\sigma^-)^\vee,\bb{1})\right).
    \]
    Since both $\sigma^+$ and $\sigma^-$ are unitary representations, this sum can be rewritten as
    \[
    \sum_{{i \in \tilde{\mathcal{T}}_{E, K}}}\left( \dim\Hom_{i(T_{E,K})}(\sigma^+,\bb{1})+\dim\Hom_{i(T_{E,K})}(\sigma^-,\bb{1})\right),
    \]
    which is equal to $1$ by \cite[Corollary~4.2]{BFGYYZ}.
\end{proof}

\begin{propo}\label{prop: nongeneric vanishing}
    If $\chi^2=1$, the representation $\pi^-$ does not have a nonzero non-generic subquotient.
\end{propo}
\begin{proof}
    Combining Proposition~\ref{prop: coinvariants of pi plus} and the previous corollary, it follows that, for all character $\psi_E$ of $U_2$ corresponding to a cubic separable $F$-algebra $E$ as above,
    \[
    \Hom_{U_2}(\pi^-,\psi_E)=0.
    \]
    Therefore, it follows from \cite[Lemma~4.13]{BS} that any non-generic subquotient of $\pi^-$ is finite-dimensional. Since $\pi^-$ is tempered by \cite[Remark~4.4]{BS}, it cannot have any finite-dimensional subquotient, which concludes the proof.
\end{proof}

Combining Proposition~\ref{prop: generic vanishing} and Proposition~\ref{prop: nongeneric vanishing}, we finally obtain the following.

\begin{theorem}\label{thm: vanishing}
    If $\chi^2=1$, the representation $\pi^-$ is zero.
\end{theorem}

\bibliographystyle{amsalpha}
\bibliography{refs}

\newcommand{\etalchar}[1]{$^{#1}$}
\providecommand{\bysame}{\leavevmode\hbox to3em{\hrulefill}\thinspace}
\providecommand{\MR}{\relax\ifhmode\unskip\space\fi MR }
\providecommand{\MRhref}[2]{%
  \href{http://www.ams.org/mathscinet-getitem?mr=#1}{#2}
}
\providecommand{\href}[2]{#2}
\begin{thebibliography}{MVW87}

\bibitem[AG17]{AG}
Hiraku Atobe and Wee~Teck Gan, \emph{Local theta correspondence of tempered
  representations and {L}anglands parameters}, Invent. Math. \textbf{210}
  (2017), no.~2, 341--415. \MR{3714507}

\bibitem[Art89]{arthur1989unipotent}
James Arthur, \emph{Unipotent automorphic representations: conjectures},
  Orbites unipotentes et repr\'esentations {-} II. Groupes $p$-adiques et
  r\'eels, Ast\'erisque, no. 171-172, Soci\'et\'e math\'ematique de France,
  1989 (en). \MR{1021499}

\bibitem[Art90]{Arthur:unipotent-motivation}
\bysame, \emph{Unipotent automorphic representations: global motivation},
  Automorphic Forms, Shimura Varieties, and L-functions (Proceedings of a
  Conference held at the University of Michigan, Ann Arbor, MI, 1988), Vol. I,
  Perspect. Math., 10, Academic Press, Boston, MA, 1990, pp.~1--75.
  \MR{1044818}

\bibitem[BFG{\etalchar{+}}]{BFGYYZ}
Neelima Borade, Jonas Franzel, Johannes Girsch, Wei Yao, Qiyao Yu, and Elad
  Zelingher, \emph{On tori periods of theta lifting}, in preparation.

\bibitem[BHLHS]{BHLS}
Petar Bakić, Aleksander Horawa, Siyan~Daniel Li-Huerta, and Naomi Sweeting,
  \emph{Global long root {A}-packets for {$G_2$}: the dihedral case}, in
  preparation.

\bibitem[BS]{BS}
Petar Baki{\'c} and Gordan Savin, \emph{Howe duality for a quasi-split
  exceptional dual pair}, to appear in Math. Annalen.

\bibitem[Gan22]{AWS}
Wee~Teck Gan, \emph{Automorphic forms and the theta correspondence}, notes for
  the Arizona Winter School 2022: Automorphic Forms Beyond {${\rm GL}_2$},
  2022.

\bibitem[GG06]{GG}
Wee~Teck Gan and Nadya Gurevich, \emph{Nontempered {A}-packets of {$G_2$}:
  liftings from {$\widetilde{\rm SL}_2$}}, Amer. J. Math. \textbf{128} (2006),
  no.~5, 1105--1185. \MR{2262172}

\bibitem[GGJ02]{GGJ}
Wee~Teck Gan, Nadya Gurevich, and Dihua Jiang, \emph{Cubic unipotent {A}rthur
  parameters and multiplicities of square integrable automorphic forms},
  Invent. Math. \textbf{149} (2002), no.~2, 225--265. \MR{1918673}

\bibitem[GR91]{GR}
Stephen~S. Gelbart and Jonathan~D. Rogawski, \emph{{$L$}-functions and
  {F}ourier-{J}acobi coefficients for the unitary group {${\rm U}(3)$}},
  Invent. Math. \textbf{105} (1991), no.~3, 445--472. \MR{1117148}

\bibitem[GS98]{GS}
Benedict~H. Gross and Gordan Savin, \emph{Motives with {G}alois group of type
  {$G_2$}: an exceptional theta-correspondence}, Compos. Math. \textbf{114}
  (1998), no.~2, 153--217. \MR{1661756}

\bibitem[GS04]{endoscopiclifting}
Wee~Teck Gan and Gordan Savin, \emph{Endoscopic lifts from {${\rm PGL}_3$} to
  {$G_2$}}, Compos. Math. \textbf{140} (2004), no.~3, 793--808. \MR{2041781}

\bibitem[GS05]{GanSavinMR}
\bysame, \emph{On minimal representations definitions and properties},
  Represent. Theory \textbf{9} (2005), 46--93. \MR{2123125}

\bibitem[GS23]{GS21}
\bysame, \emph{Howe duality and dichotomy for exceptional theta
  correspondences}, Invent. Math. \textbf{232} (2023), no.~1, 1--78.
  \MR{4557399}

\bibitem[GT16]{GT}
Wee~Teck Gan and Shuichiro Takeda, \emph{A proof of the {H}owe duality
  conjecture}, J. Amer. Math. Soc. \textbf{29} (2016), no.~2, 473--493.
  \MR{3454380}

\bibitem[HKS96]{HKS}
Michael Harris, Stephen~S. Kudla, and William~J. Sweet, \emph{Theta dichotomy
  for unitary groups}, J. Amer. Math. Soc. \textbf{9} (1996), no.~4, 941--1004.
  \MR{1327161}

\bibitem[How79]{Howe}
R.~Howe, \emph{{$\theta $}-series and invariant theory}, Automorphic forms,
  representations and {$L$}-functions ({P}roc. {S}ympos. {P}ure {M}ath.,
  {O}regon {S}tate {U}niv., {C}orvallis, {O}re., 1977), {P}art 1, Proc. Sympos.
  Pure Math., XXXIII, Amer. Math. Soc., Providence, R.I., 1979, pp.~275--285.
  \MR{546602}

\bibitem[JR97]{JR}
Dihua Jiang and Stephen Rallis, \emph{Fourier coefficients of {E}isenstein
  series of the exceptional group of type {$G_2$}}, Pacific J. Math.
  \textbf{181} (1997), no.~2, 281--314. \MR{1486533}

\bibitem[Kud94]{K}
Stephen~S. Kudla, \emph{Splitting metaplectic covers of dual reductive pairs},
  Israel J. Math. \textbf{87} (1994), no.~1-3, 361--401. \MR{1286835}

\bibitem[M{\'{i}}n08]{minguez}
Alberto M{\'{i}}nguez, \emph{Correspondance de {H}owe explicite: paires duales
  de type {II}}, Ann. Sci. \'{E}c. Norm. Sup\'{e}r. (4) \textbf{41} (2008),
  no.~5, 717--741. \MR{2504432}

\bibitem[MS97]{MS}
K.~Magaard and G.~Savin, \emph{Exceptional {$\Theta$}-correspondences. {I}},
  Compos. Math. \textbf{107} (1997), no.~1, 89--123. \MR{1457344}

\bibitem[Mui97]{M}
Goran Mui\'{c}, \emph{The unitary dual of {$p$}-adic {$G_2$}}, Duke Math. J.
  \textbf{90} (1997), no.~3, 465--493. \MR{1480543}

\bibitem[MVW87]{MVW}
Colette M{\oe}glin, Marie-France Vign\'{e}ras, and Jean-Loup Waldspurger,
  \emph{Correspondances de {H}owe sur un corps {$p$}-adique}, Lecture Notes in
  Mathematics, vol. 1291, Springer-Verlag, Berlin, 1987. \MR{1041060}

\bibitem[Rob96]{Roberts}
Brooks Roberts, \emph{The theta correspondence for similitudes}, Israel J.
  Math. \textbf{94} (1996), 285--317. \MR{1394579}

\bibitem[Rog90]{Rog90}
Jonathan~D. Rogawski, \emph{Automorphic representations of unitary groups in
  three variables}, Annals of Mathematics Studies, vol. 123, Princeton
  University Press, Princeton, NJ, 1990. \MR{1081540}

\bibitem[Rog92]{R}
\bysame, \emph{The multiplicity formula for {$A$}-packets}, The zeta functions
  of {P}icard modular surfaces, Univ. Montr\'{e}al, Montreal, QC, 1992,
  pp.~395--419. \MR{1155235}

\bibitem[SZ15]{SZ}
Binyong Sun and Chen-Bo Zhu, \emph{Conservation relations for local theta
  correspondence}, J. Amer. Math. Soc. \textbf{28} (2015), no.~4, 939--983.
  \MR{3369906}

\bibitem[Wal90]{Waldspurger}
J.-L. Waldspurger, \emph{D\'{e}monstration d'une conjecture de dualit\'{e} de
  {H}owe dans le cas {$p$}-adique, {$p\neq 2$}}, Festschrift in honor of {I}.
  {I}. {P}iatetski-{S}hapiro on the occasion of his sixtieth birthday, {P}art
  {I} ({R}amat {A}viv, 1989), Israel Math. Conf. Proc., vol.~2, Weizmann,
  Jerusalem, 1990, pp.~267--324. \MR{1159105}

\end{thebibliography}

\end{document}